\documentclass{article}

\usepackage{amsmath}
\usepackage{amssymb}
\usepackage{amsthm}
\usepackage{amsfonts}
\usepackage{cite}
\usepackage{geometry}
\usepackage{graphicx}
\usepackage{hyperref}
\usepackage{epstopdf}
\usepackage{algorithmic}

\newtheorem{theorem}{Theorem}[section]
\newtheorem{proposition}[theorem]{Proposition}
\newtheorem{lemma}[theorem]{Lemma}
\newtheorem{remark}[theorem]{Remark}
\newtheorem{definition}[theorem]{Definition}
\newtheorem{assumption}[theorem]{Assumption}

\title{The Global Well-posedness for Master Equations of Mean Field Games of Controls}
\author{Shuhui Liu\thanks{Department of Applied Mathematics, The Hong Kong Polytechnic University, Hong Kong SAR, China 
  (\href{mail to:shuhui.liu@polyu.edu.hk}{shuhui.liu@polyu.edu.hk}).}
\and Xintian Liu\thanks{Department of Mathematics, City University of Hong Kong, Hong Kong SAR, China
  (\href{mail to:xintialiu2-c@my.cityu.edu.hk}{xintialiu2-c@my.cityu.edu.hk}).}
\and Chenchen Mou\thanks{Department of Mathematics, City University of Hong Kong, Hong Kong SAR, China 
  (\href{mail to:chencmou@cityu.edu.hk}{chencmou@cityu.edu.hk}).}
\and Defeng Sun\thanks{Department of Applied Mathematics, The Hong Kong Polytechnic University, Hong Kong SAR, China  
  (\href{mail to:defeng.sun@polyu.edu.hk}{defeng.sun@polyu.edu.hk}).}}
\date{}

\usepackage{amsopn}
\DeclareMathOperator{\E}{\mathbb{E}}
\DeclareMathOperator{\R}{\mathbb{R}}
\DeclareMathOperator{\tr}{tr}
\DeclareMathOperator{\divergence}{div}

\begin{document}
\maketitle

\begin{abstract}
In this manuscript, we establish the global well-posedness for master equations of mean field games of controls, where the interaction is through the joint law of the state and control. Our results are proved under two different conditions: the Lasry-Lions monotonicity and the displacement $\lambda$-monotonicity, both considered in their integral forms. We provide a detailed analysis of both the differential and integral versions of these monotonicity conditions for the corresponding nonseparable Hamiltonian and examine their relation. The proof of global well-posedness relies on the propagation of these monotonicity conditions in their integral forms and a priori uniform Lipschitz continuity of the solution with respect to the measure variable.
\end{abstract}


\noindent\textbf{MSCcodes.} 35R15, 49N80, 49Q22, 60H30, 91A16, 93E20

\section{Introduction}
The mean field games (MFG) theory was developed by Lasry-Lions \cite{lasry2006jeux-i,lasry2006jeux-ii,lasry2007mean} and Huang-Caines-Malham\'{e} \cite{huang2007large, huang2006large} independently at about the same time. The problems consider the limit behavior of differential games, with a large number of players that have little influence on the whole system symmetrically but the collective behaviors are significant. We refer to \cite{bensoussan2013mean,cardaliaguet2010notes,achdou2020introduction,carmona2018probabilistic} for materials on MFG theory. The master equation, firstly introduced in \cite{lions2007theorie}, is an infinite-dimensional hyperbolic equation stated on the space of probability measures. Provided there is a unique mean field equilibrium (MFE), the master equation can be used to characterize the minimal cost of MFG under a given background. The master equation is quite useful in studying the limit theory of MFG. The interested readers could see \cite{cardaliaguet2019master} for the convergence problem, \cite{delarue2019central} for central limit theorem and \cite{delarue2020deviation} for large deviations and concentration of measure.
\par In a classical MFG model, the interaction is only through the law of state. As a generalization, the mean field games of controls (MFGC), also referred to as extended mean field games in the earlier literature, allow the mean field interaction to include the law of control. As a consequence, the joint law of state and control may appear in the state equation and the cost functional in MFGC. Some recent studies focus on the MFGC system, which is a system of forward-backward stochastic partial differential equation (FBSPDE) related to the MFGC problem as follows:
\begin{equation}
\label{MFGC intro}
\left\{\begin{aligned}
    &d\mu_t(x)=\Big[\frac{\widehat{\beta}^2}{2}\mathrm{tr}(\partial_{xx}\mu_t(x))-\divergence\big(\mu_t(x)\partial_p H(x,\partial_x u(t,x),\rho_t)\big)\Big]\,dt-\beta\partial_x\mu_t(x)\cdot dB_t^0,\\
    &du(t,x)=-\Big[\frac{\widehat{\beta}^2}{2}\tr(\partial_{xx}u(t,x))+\beta\tr(\partial_x v(t,x))+H(x,\partial_x u(t,x),\rho_t)\Big]dt\\
    &\qquad\qquad\quad+v(t,x)\cdot dB_t^0,\\
    &\rho_t=\big(id,\partial_p H(\cdot,\partial_x u(t,\cdot),\rho_t)\big)\#\mu_t,\quad\mu_{t_0}=\mathcal{L}_{\xi},\quad u(T,x)=G(x,\mu_T).
\end{aligned}\right.
\end{equation}
Here the time horizon is $[t_0,T]$; the Hamiltonian $H:\R^d\times\R^d\times\mathcal{P}_2(\R^{2d})\rightarrow\R$ and the terminal cost $G:\R^d\times\mathcal{P}_2(\R^d)\rightarrow\R$ serve as the data of the MFGC; $\beta$ stands for the intensity of the common noise, while the intensity of the idiosyncratic noise is assumed to be $1$ (nondegenerate), and we denote $\widehat{\beta}^2:=1+\beta^2$. We refer to \cite{cardaliaguet2018mean,gomes2014existence,gomes2016extended,kobe2022} for studies on such MFGC system without common noise under Lasry-Lions monotonicity condition. Very recently, \cite{jackson2025quantitative} studies the convergence of Nash equilibria from $N$-player games to MFGC under displacement monotonicity condition. 
\par However, there are few studies on the master equation of MFGC. In this manuscript, our aim is to study the global well-posedness for the following master equation of MFGC:
\begin{equation}
\left\{\begin{aligned}
    &\mathcal{L}V(t,x,\mu):=\partial_t V+\frac{\widehat{\beta}^2}{2}\tr(\partial_{xx}V)+\widehat{H}(x,\partial_x V,\mathcal{L}_{(\xi,\partial_x V(t,\xi,\mu))})+\mathcal{M}V=0,\\
    &V(T,x,\mu)=G(x,\mu),
\end{aligned}\right.
\end{equation}
where $\xi$ is a random variable with law $\mu$, and $\mathcal{M}$ is the nonlocal operator    
\begin{equation}
\begin{aligned}
    &\mathcal{M}V(t,x,\mu)\\
    :=&\tr\int_{\R^d}\int_{\R^d}\Big[\frac{\widehat{\beta}^2}{2}\partial_{\tilde{x}\mu}V(t,x,\mu,\tilde{x})+\beta^2\partial_{x\mu}V(t,x,\mu,\tilde{x})+\frac{\beta^2}{2}\partial_{\mu\mu}V(t,x,\mu,\bar{x},\tilde{x})\\
    &+\partial_{\mu}V(t,x,\mu,\tilde{x})\cdot\partial_p\widehat{H}\big(\tilde{x},\partial_x V(t,\tilde{x},\mu),\mathcal{L}_{(\xi,\partial_{x}V(t,\xi,\mu))}\big)\Big]\,\mu(d\tilde{x})\mu(d\bar{x}).
\end{aligned}
\end{equation}
We emphasize that the $\widehat{H}$ here is not the standard Hamiltonian $H$ as in (\ref{MFGC intro}). In fact, note that the $\rho$ in (\ref{MFGC intro}) is given via $\rho=\mathcal{L}_{(\xi,\partial_p H(\xi,\eta,\rho))}$ for random variables $\xi,\eta$, which can induce a fixed-point mapping $\rho=\Phi(\mathcal{L}_{(\xi,\eta)})$ under appropriate conditions. Here, we denote $\widehat{H}(\cdot,\cdot,\rho):=H(\cdot,\cdot,\Phi(\rho))$. The derivative with respect to the measure variable is understood in the sense of the Lions derivative.
\par It is known that, in standard MFG, the monotonicity condition is crucial to guarantee both the uniqueness of the MFE and the global well-posedness of the master equation. In the literature, two main types of monotonicity conditions for the master equation have been provided: Lasry-Lions monotonicity and displacement monotonicity. The Lasry-Lions monotonicity, introduced in \cite{lions2007theorie}, was firstly used to study the global well-posedness of the master equation (see e.g. \cite{cardaliaguet2019master, chassagneux2014probabilistic}). We mention that in these results, the Hamiltonian is assumed to be separable in the momentum and measure variables. The displacement monotonicity, introduced in \cite{ahuja2016wellposedness}, is also an appropriate condition. With this monotonicity condition, the global well-posedness of master equation was shown with nonseparable Hamiltonian in \cite{mou2022displacement}. Moreover, \cite{bansil2025degenerate} presented a more general result in the lack of idiosyncratic noise. The displacement monotonicity can be further weakened to the displacement $\lambda$-monotonicity (see e.g. \cite{mou2022propagation}). In contrast to the previous two, the anti-monotonicity is imposed in the opposite direction, under which the global well-posedness of the master equation was shown in \cite{mou2022anti}. Without monotonicity conditions, multiple MFEs may exist. The interested reader can see e.g. \cite{bayraktar2020non, cecchin2022selection, cecchin2022weak, mou2024minimal} for details.
\par In \cite{mou2022displacement}, a strategy was developed for establishing the global well-posedness of the master equation in its differential form. This approach consists of three main steps:
\begin{itemize}
    \item [{Step 1:}] Assume appropriate monotonicity conditions on the data and show that the monotonicity can be propagated along any classical solution $V$ to the master equation;
    \item[Step 2:] Show that the monotonicity of the solution $V$ implies an a priori uniform Lipschitz continuity in the measure variable;
    \item[Step 3:] Conclude the global well-posedness from the local well-posedness result and the uniform Lipschitz continuity.
\end{itemize}
\par In this manuscript, we focus on the Lasry-Lions monotonicity and displacement 
$\lambda$-monotonicity in their integral forms. We mention that the propagation of monotonicity (Step 1) has been proved in \cite{mou2022propagation}, where the monotonicities were provided in their differential forms. The proof there was based on differentiating the master equation, which required higher regularity assumptions on data. Moreover, the a priori $W_1$-Lipschitz estimate of $V$ (Step 2) and the global well-posedness of master equation (Step 3) were absent there. In contrast, we first show the propagation of Lasry-Lions monotonicity and displacement 
$\lambda$-monotonicity in their integral forms via the MFGC system instead. We then prove the a priori $W_1$-Lipschitz estimate of $V$ by its obtained monotonicity in respective integral forms. As a result, we establish the global well-posedness of the master equation with less regular data. 
\par It is known that the $W_1$-Lipschitz continuity is crucial to the global well-posedness for master equation. When $V$ is assumed to be displacement $\lambda$-monotone, we first establish the $W_2$-Lipschitz continuity of $\partial_x V$ in the measure variable. Furthermore, by a pointwise representation formula of $\partial_{x\mu}V$, we show that the $W_2$-Lipschitz continuity implies the $W_1$-Lipschitz continuity of $\partial_x V$ in the measure variable, which is equivalent to the boundedness of $\partial_{x\mu}V$. In the case where $V$ is Lasry-Lions monotone, the $W_1$-Lipschitz continuity can be obtained directly. The local well-posedness result for the master equation then follows from the local well-posedness theory for the Mckean-Vlasov forward-backward stochastic differential equation (FBSDE) and its stability. Once we obtain the local well-posedness of the master equation and the uniform bound for $\partial_{x\mu}V$, we can construct the global solution by extending the local solution backward in time.
\par We summarize the main contributions of this manuscript as follows:
\begin{itemize}
    \item Using interpolation with appropriate geodesic lines, we demonstrate the relation for the nonseparable Hamiltonian between the differential and integral forms of both Lasry-Lions and displacement $\lambda$-monotonicity conditions. In particular, the differential form of Lasry-Lions monotonicity implies our new integral form so the integral form is raised as a weaker assumption. The integral and differential forms of displacement $\lambda$-monotonicity are shown to be equivalent. Furthermore, we show the propagation of these monotonicity conditions in their integral forms through the MFGC system.
    \item We prove the a priori uniform  $W_1$-Lipschitz estimate for the solution under both Lasry-Lions monotonicity and displacement $\lambda$-monotonicity conditions through the propagation of monotonicity conditions and the pointwise representation formula of the solution. This $W_1$-Lipschitz estimate serves as a crucial step towards the global well-posedness of the  master equation.
    \item We establish the global well-posedness of the master equation for MFGC with nonseparable Hamiltonian under Lasry-Lions monotonicity and displacement $\lambda$-monotonicity conditions, respectively. To our best knowledge, this is the first general global well-posedness result for MFGC master equation. 
\end{itemize}
\par The remainder of this manuscript is organized as follows. In Section 2, we present the formulation of the MFGC, the corresponding master equation, and the crucial monotonicity conditions. In section 3, we first introduce the Lasry-Lions monotonicity and displacement $\lambda$-monotonicity conditions for nonseparable Hamiltonian in their integral forms. Moreover, we examine the relation between the integral and differential forms of these monotonicity conditions. In Section 4, we investigate the propagation of monotonicity conditions through the MFGC system. Section 5 establishes the key a priori uniform Lipschitz continuity in the measure variable. Finally, in Section 6, we prove the global well-posedness of the master equation for MFGC under the monotonicity conditions.

\section{Preliminaries}
\subsection{The product probability spaces and function spaces}
We follow the setting for mean field games of controls (MFGC) in \cite{mou2022propagation}. Throughout the paper let $T$ be a fixed finite time horizon. Let $(\Omega_0,\mathbb{F}^0,\mathbb{P}_0)$ and $(\Omega_1,\mathbb{F}^1,\mathbb{P}_1)$ be two filtered probability spaces on which there are defined $d$-dimensional Brownian motions $B^0$ and $B$, respectively. For $\mathbb{F}^i=\{\mathcal{F}_t^i\}_{0\leq t\leq T}$, $i=0,1$, we assume $\mathcal{F}_t^0=\mathcal{F}_t^{B^0}$, $\mathcal{F}_t^1=\mathcal{F}_0^1\vee\mathcal{F}_t^B$, and $\mathbb{P}_1$ has no atom in $\mathcal{F}_0^1$ so it can support any measure on $\R^d$ with finite second order moment. Consider the product space
\begin{equation}
    \Omega:=\Omega_0\times\Omega_1,\quad \mathbb{F}=\{\mathcal{F}_t\}_{0\leq t\leq T}:=\{\mathcal{F}_t^0\otimes\mathcal{F}_t^1\}_{0\leq t\leq T},\quad \mathbb{P}:=\mathbb{P}_0\otimes\mathbb{P}_1,\quad \E:=\E^{\mathbb{P}}.
\end{equation}
In particular, $\mathcal{F}_t:=\sigma(A_0\times A_1: A_0\in\mathcal{F}_t^0, A_1\in\mathcal{F}_t^1)$ and $\mathbb{P}(A_0\times A_1)=\mathbb{P}_0(A_0)\mathbb{P}_1(A_1)$. We shall automatically extend $B^0,B,\mathbb{F}^0,\mathbb{F}^1$ to the product space in the obvious sense, but using the same notation. Note that $B^0$ and $B^1$ are independent $\mathbb{P}$-Brownian motions and are independent of $\mathcal{F}_0$.
\par It is convenient to introduce another filtered probability space $(\tilde{\Omega}_1,\tilde{\mathbb{F}}^1,\tilde{B},\tilde{\mathbb{P}}_1)$ in the same manner as $(\Omega_1,\mathbb{F}^1,B,\mathbb{P}_1)$, and consider the larger filtered probability space given by
\begin{equation}
\label{product space}
    \tilde{\Omega}:=\Omega\times\tilde{\Omega}_1,\quad \tilde{\mathbb{F}}=\{\tilde{\mathcal{F}}_t\}_{0\leq t\leq T}:=\{\mathcal{F}_t\otimes\tilde{\mathcal{F}}_t^1\}_{0\leq t\leq T},\quad \tilde{\mathbb{P}}:=\mathbb{P}\otimes\tilde{\mathbb{P}}_1,\quad \tilde{\E}:=\E^{\tilde{\mathbb{P}}}.
\end{equation}
Given an $\mathcal{F}_t$-measurable random variable $\xi=\xi(\omega^0,\omega^1)$, we say $\tilde{\xi}=\tilde{\xi}(\omega^0,\tilde{\omega}^1)$ is a conditionally independent copy of $\xi$ if, for each $\omega^0$, the $\mathbb{P}_1$-distribution of $\xi(\omega^0,\cdot)$ is equal to the $\tilde{\mathbb{P}}_1$-distribution of $\tilde{\xi}(\omega^0,\cdot)$. That is, conditional on $\mathcal{F}_t^0$, by extending to $\tilde{\Omega}$ the random variables $\xi$ and $\tilde{\xi}$ are conditionally independent and have the same conditional distribution under $\tilde{\mathbb{P}}$. Note that, for any appropriate deterministic function $\varphi$,
\begin{equation}
\begin{gathered}
    \tilde{\E}_{\mathcal{F}_t^0}[\varphi(\xi,\tilde{\xi})](\omega^0)=\E^{\mathbb{P}_1\otimes\tilde{\mathbb{P}}_1}\big[\varphi\big(\xi(\omega^0,\cdot),\tilde{\xi}(\omega^0,\tilde{\cdot})\big)\big],\quad\mathbb{P}_0-\text{a.e. }\omega^0;\\
    \tilde{\E}_{\mathcal{F}_t}[\varphi(\xi,\tilde{\xi})](\omega^0,\omega^1)=\E^{\tilde{\mathbb{P}}_1}\big[\varphi\big(\xi(\omega^0,\omega^1),\tilde{\xi}(\omega^0,\tilde{\cdot})\big)\big],\quad\mathbb{P}-\text{a.e. }(\omega^0,\omega^1).
\end{gathered}
\end{equation}
Here $\E^{\tilde{\mathbb{P}}_1}$ is the expectation on $\tilde{\omega}^1$, and $\E^{\mathbb{P}_1\times\tilde{\mathbb{P}}_1}$ is on $(\omega^1,\tilde{\omega}^1)$. Throughout the paper, we will use the probability space $(\Omega,\mathbb{F},\mathbb{P})$. However, when conditionally independent copies of random variables or processes are needed, we will tacitly use the extension to the larger space $(\tilde{\Omega},\tilde{\mathbb{F}},\tilde{\mathbb{P}})$ without mentioning. When we need two conditionally independent copies, we similarly define $(\bar{\tilde{\Omega}},\bar{\tilde{\mathbb{F}}},\bar{\tilde{\mathbb{P}}},\bar{\tilde{\E}})$ as (\ref{product space}).
\par For any dimension $d$ and any constant $p\geq 1$, let $\mathcal{P}(\R^d)$ denote the set of probability measures on $\R^d$, and $\mathcal{P}_p(\R^d)$ the subset of probability measures with finite $p$-th moment, equipped with the $p$-Wasserstein distance $W_p$. Moreover, for any sub-$\sigma$-algebra $\mathcal{G}\subset\mathcal{F}_T$, $\mathbb{L}^p(\mathcal{G})$ denotes the set of $\R^d$-valued, $\mathcal{G}$-measurable, and $p$-integrable random variables. For any $\mu\in\mathcal{P}_p(\R^d)$, let $\mathbb{L}^p(\mathcal{G};\mu)$ denote the set of $\xi\in\mathbb{L}^p(\mathcal{G})$ with law $\mathcal{L}_{\xi}=\mu$. Similarly, for any sub-filtration $\mathbb{G}\subset\mathbb{F}$, $\mathbb{L}(\mathbb{G};\R^d)$ denotes the set of $\mathbb{G}$-progressively measurable $\R^d$-valued processes.
\par Let $\mathcal{C}^0(\mathcal{P}_2(\R^d))$ denote the set of $W_2$-continuous functions $U:\mathcal{P}_{2}(\R^d)\rightarrow\R$. For a continuous function $U\in\mathcal{C}^0(\mathcal{P}_2(\R^d))$, we use $\partial_{\mu}U: (\mu,\tilde x)\in\mathcal{P}_2(\R^d)\times\R^d\rightarrow\R$ to denote its Lions derivative. We say $U\in \mathcal{C}^1(\mathcal{P}_2(\R^d))$ if $\partial_{\mu}U$ exists and is continuous on $\mathcal{P}_2(\R^d)\times\R^d$. We can further define the second order derivative $\partial_{\mu\mu}U:(\mu,\tilde{x},\bar{x})\in \mathcal{P}_2(\R^d)\times\R^d\times\R^d\to \R$, and we say $U\in \mathcal{C}^2(\mathcal{P}_2(\R^d))$ if $\partial_{\mu}U$, $\partial_{\tilde{x}\mu}U$ and $\partial_{\mu\mu}U$ exist and are continuous. Similarly, let $\mathcal{C}^0(\R^d\times\mathcal{P}_2(\R^d))$ denote the set of continuous functions $U:\R^d\times\mathcal{P}_2(\R^d)\rightarrow\R^d$. We say $U\in\mathcal{C}^1(\R^d\times\mathcal{P}_2(\R^d))$ if $\partial_x U,\partial_{\mu}U$ exist and are continuous. We say $U\in\mathcal{C}^2(\R^d\times\mathcal{P}_2(\R^d))$ if all the derivatives $\partial_x U,\partial_{xx}U,\partial_{\mu}U,\partial_{x\mu}U,\partial_{\tilde{x}\mu}U,\partial_{\mu\mu}U$ exist and continuous. Moreover, $\mathcal{C}_b^1(\R^d\times\mathcal{P}_2(\R^d))$ denotes the set of functions $U\in\mathcal{C}^1(\R^d\times\mathcal{P}_2(\R^d))$ with bounded $\partial_x U,\partial_{\mu}U$; $\mathcal{C}_c^1(\R^d)$ denotes the set of functions $U\in\mathcal{C}^1(\R^d)$ with compact support, and $\mathcal{C}_c^{\infty}(\R^d)$ denotes the set of smooth functions with compact support. We also use the following notations: for $R>0$, $B_R:=\{p\in\R^d:|p|\leq R\},\,B_R^o:=\{p\in\R^d:|p|<R\}$.
\par Given $t_0\in[0,T]$, denote $B_t^{t_0}:=B_t-B_{t_0},\,B_t^{0,t_0}:=B_t^0-B_{t_0}^0,\,t\in[t_0,T]$. Let $\mathcal{A}_{t_0}$ denote the set of admissible controls $\alpha:[t_0,T]\times\R^d\times \mathcal{C}([t_0,T];\R^d)\rightarrow\R^d$ which are progressively measurable and adapted in the path variable and square integrable. Let $\mathbb{L}^2(\mathbb{F}^{B^{0,t_0}};\mathcal{P}_2(\R^{2d}))$ denote the set of $\mathbb{F}^{B^{0,t_0}}$-progressively measurable stochastic probability measure flows $\{\nu.\}=\{\nu_t\}_{t\in[t_0,T]}\subset\mathcal{P}_2(\R^{2d})$.
\subsection{Mean field games of controls}
The data of our MFGC is given by
\[
b:\R^{2d}\times\mathcal{P}_2(\R^{2d})\rightarrow\R^d;\quad f:\R^{2d}\times\mathcal{P}_2(\R^{2d})\rightarrow\R;\quad G:\R^d\times\mathcal{P}_2(\R^d)\rightarrow\R
\]
and $\beta\in[0,\infty)$. For simplicity, we assume that $b$ and $f$ do not depend on time.
\par Given $t_0\in[0,T]$, $x\in\R^d$, $\alpha\in\mathcal{A}_{t_0}$, and $\{\nu.\}\in\mathbb{L}^2(\mathbb{F}^{B^{0,t_0}};\mathcal{P}_2(\R^{2d}))$, the state of the agent satisfies the following controlled stochastic differential equation (SDE) on $[t_0,T]$:
\begin{equation}\label{controlled SDE}
\begin{gathered}
    X_t^{\{\nu.\},\alpha}=x+\int_{t_0}^t b(X_s^{\{\nu.\},\alpha},\alpha_s,\nu_s)\,ds+B_t^{t_0}+\beta B_t^{0,t_0};\\
    \text{where }X^{\{\nu.\},\alpha}=X^{t_0,\{\nu.\};x,\alpha},\quad \alpha_t:=\alpha(t,X_t^{\{\nu.\},\alpha},B_{[t_0,t]}^{0,t_0}).
\end{gathered}
\end{equation}
Let $\pi_1\#\nu_T$ denote the first $\R^d$ marginal measure of $\nu_t$. The expected cost is given by
\begin{equation}
    J(t_0,x;\{\nu.\},\alpha):=\E\Big[G(X_T^{\{\nu.\},\alpha},\pi_1\#\nu_T)+\int_{t_0}^T f(X_t^{\{\nu.\},\alpha},\alpha_t,\nu_t)\,dt\Big].
\end{equation}
\begin{definition}
\label{MFE}
    For any $(t,\mu)\in[0,T]\times\mathcal{P}_2(\R^d)$, we say $(\alpha^*,\{\nu.^*\})\in\mathcal{A}_t\times\mathbb{L}^2(\mathbb{F}^{B^{0,t}};\mathcal{P}_2(\R^{2d}))$ is a mean field equilibrium (MFE) at $(t,\mu)$ if
    \begin{equation}
    \begin{gathered}
        J(t,x;\{\nu.^*\},\alpha^*)=\inf_{\alpha\in\mathcal{A}_t}J(t,x;\{\nu.^*\},\alpha),\quad\text{for }\mu-\text{a.e. }x\in\R^d;\\
        \pi_1\#\nu_t^*=\mu,\quad \nu_s^*:=\mathcal{L}_{(X_s^*,\alpha^*(s,X_s^*,B_{[t,s]}^{0,t}))|\mathcal{F}_s^0},\quad \text{where}\\
        X_s^*=\xi+\int_t^s b(X_r^*,\alpha^*(r,X_r^*,B_{[t,r]}^{0,t}),\nu_r^*)\,dr+B_s^t+\beta B_s^{0,t},\quad\xi\in\mathbb{L}^2(\mathcal{F}_t^1,\mu).
    \end{gathered}
    \end{equation}
\end{definition}
\par When there is a unique MFE for each $(t,\mu)\in[0,T]\times\mathcal{P}_2(\R^d)$, denoted as $(\alpha^*(t,\mu;\cdot),\{\nu.^*(t,\mu)\})$, then the game problem leads to the following value function for the agent:
\begin{equation}\label{value function}
    V(t,x,\mu):=J(t,x;\{\nu.^*(t,\mu)\},\alpha^*(t,\mu;\cdot))\quad\text{for any }x\in\R^d.
\end{equation}
Our goal is to study the master equation which is formally satisfied by the value function $V(t,x,\mu)$. We define the Hamiltonian: for $(x,p,\nu)\in\R^d\times\R^d\times\mathcal{P}_2(\R^{2d})$,
\begin{equation}
    H(x,p,\nu):=\inf_{a\in\R^d}h(x,p,\nu,a),\quad h(x,p,\nu,a):=p\cdot b(x,a,\nu)+f(x,a,\nu).
\end{equation}
\begin{assumption}
\label{fixed point}
    (i) The Hamiltonian $H$ has a unique minimizer $\alpha^*=\phi(x,p,\nu)$, namely
    \begin{equation}
        H(x,p,\nu)=h(x,p,\nu,\phi(x,p,\nu)).
    \end{equation}
    (ii) For any $\xi\in\mathbb{L}^2(\mathcal{F}_0^1)$ and $\eta\in\mathbb{L}^2(\sigma(\xi))$, the following mapping on $\mathcal{P}_2(\R^{2d})$:
    \begin{equation}
        \mathcal{I}^{\xi,\eta}(\nu):=\mathcal{L}_{(\xi,\phi(\xi,\eta,\nu))}
    \end{equation}
    has a unique fixed point $\nu^*:\mathcal{I}^{\xi,\eta}(\nu^*)=\nu^*$, and we shall denote it as $\Phi(\mathcal{L}_{(\xi,\eta)})$.
\end{assumption}

\subsection{The master equation and the MFGC system}
From Assumption \ref{fixed point}, we can denote
\begin{equation}
    \widehat{H}(x,p,\rho):=H(x,p,\Phi(\rho)),\quad (x,p,\rho)\in\R^d\times\R^d\times\mathcal{P}_2(\R^{2d}).
\end{equation}
Note that $\Phi$ is a measure valued functions and the associated derivatives should be understood in the sense of \cite[Section 2.1]{mou2022propagation}.
\par By the fixed point argument, from Definition \ref{MFE} we have the McKean-Vlasov SDE
\begin{equation}
\label{SDE}
\begin{gathered}
    X_t^{\xi}=\xi+\int_{t_0}^t\partial_p\widehat{H}\big(X_s^{\xi},\partial_x V(s,X_s^{\xi},\mu_s),\rho_s\big)\,ds+B_t^{t_0}+\beta B_t^{0,t_0},\\
    \text{where }\mu_s:=\mathcal{L}_{X_s^{\xi}|\mathcal{F}_s^0},\quad \rho_s:=\mathcal{L}_{(X_s^{\xi},\partial_x V(s,X_s^{\xi},\mu_s))|\mathcal{F}_s^0}.
\end{gathered}
\end{equation}
On the other hand, it follows from the standard stochastic control theory that, for given $t_0,\mu$, the optimization in Definition \ref{MFE} is associated with the following backward stochastic differential equation (BSDE):
\begin{equation}
\label{BSDE}
\begin{gathered}
    Y_t^{\xi}=G(X_T^{\xi},\mu_T)-\int_t^T Z_s^{\xi}\cdot dB_s-\int_t^T Z_s^{0,{\xi}}\cdot dB_s^0\\
    +\int_t^T\big[\widehat{H}(\cdot)-\partial_x V(s,X_s^{\xi},\mu_s)\cdot\partial_p\widehat{H}(\cdot)\big]\big(X_s^{\xi},\partial_x V(s,X_s^{\xi},\mu_s),\rho_s\big)\,ds.
    \end{gathered}
\end{equation}
The above SDE (\ref{SDE}) and BSDE (\ref{BSDE}) form the forward-backward stochastic differential equation (FBSDE) system for the MFGC (\ref{controlled SDE})-(\ref{value function}). In particular, we have
\begin{equation}
    Y_t^{\xi}=V(t,X_t^{\xi},\mu_t),\quad Z_t^{\xi}=\partial_x V(t,X_t^{\xi},\mu_t).
\end{equation}
\par Applying It\^{o}'s formula to $V(t,X_t^{\xi},\mu_t)$ and then comparing with the BSDE (\ref{BSDE}), we can derive the master equation: for independent copies $\xi,\tilde{\xi},\bar{\xi}$ with law $\mu$,
\begin{equation}
\label{master equation}
\begin{gathered}
    \mathcal{L}V(t,x,\mu):=\partial_t V+\frac{\widehat{\beta}^2}{2}\tr(\partial_{xx}V)+\widehat{H}(x,\partial_x V,\mathcal{L}_{(\xi,\partial_x V(t,\xi,\mu))})+\mathcal{M}V=0,\\
    V(T,x,\mu)=G(x,\mu),\quad\text{where}\\
    \mathcal{M}V(t,x,\mu):=\tr\Big(\bar{\tilde{\E}}\Big[\frac{\widehat{\beta}^2}{2}\partial_{\tilde{x}\mu}V(t,x,\mu,\tilde{\xi})+\beta^2\partial_{x\mu}V(t,x,\mu,\tilde{\xi})+\frac{\beta^2}{2}\partial_{\mu\mu}V(t,x,\mu,\bar{\xi},\tilde{\xi})\\
    +\partial_{\mu}V(t,x,\mu,\tilde{\xi})\cdot\partial_p\widehat{H}\big(\tilde{\xi},\partial_x V(t,\tilde{\xi},\mu),\mathcal{L}_{(\xi,\partial_{x}V(t,\xi,\mu))}\big)\Big]\Big),\quad\text{and }\widehat{\beta}^2:=1+\beta^2.
\end{gathered}
\end{equation}
\begin{definition}
    \label{classical solution}
    We say $V:[0,T]\times\R^d\times\mathcal{P}_2(\R^d)$ is a classical solution to the master equation if $V$ satisfies the master equation (\ref{master equation}) pointwisely, and all the derivatives involved in the equation exist and are continuous.
\end{definition}
Note that the assumption that $V$ is a classical solution already implies the regularities of $G$. Hence we may not include Assumption \ref{G regularity} again when introducing the a priori estimates later.
\par Given the $\rho$ in the FBSDE system (\ref{SDE})-(\ref{BSDE}), we can further consider the following decoupled FBSDE on $[t_0,T]$:
\begin{equation}
\label{decoupled FBSDE}
\left\{\begin{aligned}
    &X_t^x=x+B_t^{t_0}+\beta B_t^{0,t_0},\\
    &Y_t^{x,\xi}=G(X_T^x,\mu_T)+\int_t^T\widehat{H}(X_s^x,Z_s^{x,\xi},\rho_s)ds-\int_t^T Z_s^{x,\xi}\cdot dB_s-\int_t^T Z_s^{0,x,\xi}\cdot dB_s^0.
\end{aligned}\right.
\end{equation}
The above FBSDE connects to the master equation by
\begin{equation}
    Y_t^{x,\xi}=V(t,X_t^x,\mu_t),\quad Z_t^{x,\xi}=\partial_x V(t,X_t^x,\mu_t).
\end{equation}
\par Equivalently, one may also consider the following forward-backward stochastic partial differential equation (FBSPDE) known as the MFGC system, with the relation $u(t,x)=V(t,x,\mu_t)$:
\begin{equation}
\label{FBSPDE}
\left\{\begin{aligned}
    &d\mu_t(x)=\Big[\frac{\widehat{\beta}^2}{2}\mathrm{tr}(\partial_{xx}\mu_t(x))-\divergence\big(\mu_t(x)\partial_p \widehat{H}(x,\partial_x u(t,x),\rho_t)\big)\Big]\,dt-\beta\partial_x\mu_t(x)\cdot dB_t^0,\\
    &du(t,x)=-\Big[\frac{\widehat{\beta}^2}{2}\tr(\partial_{xx}u(t,x))+\beta\tr(\partial_x v(t,x))+\widehat{H}(x,\partial_x u(t,x),\rho_t)\Big]dt\\
    &\qquad\qquad\quad+v(t,x)\cdot dB_t^0,\\
    &\rho_t=\big(id,\partial_x u(t,\cdot)\big)\#\mu_t,\quad\mu_{t_0}=\mathcal{L}_{\xi},\quad u(T,x)=G(x,\mu_T).
\end{aligned}\right.
\end{equation}
\subsection{Monotonicities and main assumptions}
\par Next we introduce the two types of monotonicity conditions: Lasry-Lions monotonicity and displacement $\lambda$-monotonicity in their integral forms. We call them integral forms because the inequalities (\ref{LL monotone}) and (\ref{disp monotone}) written by expectation can be easily rewritten by integral.
\begin{definition}
    Assume $U\in\mathcal{C}^0(\R^d\times\mathcal{P}_2(\R^d))$. We say $U$ is Lasry-Lions monotone if
    \begin{equation}
    \label{LL monotone}
        \E\big[U(\xi^1,\mathcal{L}_{\xi^1})+U(\xi^2,\mathcal{L}_{\xi^2})-U(\xi^1,\mathcal{L}_{\xi^2})-U(\xi^2,\mathcal{L}_{\xi^1})\big]\geq 0,\quad\forall\xi^1,\xi^2\in\mathbb{L}^2(\mathcal{F}_T^1).
    \end{equation}
\end{definition}

\begin{definition}
    Assume $U,\partial_x U\in\mathcal{C}^0(\R^d\times\mathcal{P}_2(\R^d))$. For any $\lambda\geq 0$, we say $U$ is displacement $\lambda$-monotone if
    \begin{equation}
    \label{disp monotone}
        \E\big[\langle\partial_x U(\xi^1,\mathcal{L}_{\xi^1})-\partial_x U(\xi^2,\mathcal{L}_{\xi^2}),\xi^1-\xi^2\rangle+\lambda |\xi^1-\xi^2|^2\big]\geq 0,\quad\forall\xi^1,\xi^2\in\mathbb{L}^2(\mathcal{F}_T^1).
    \end{equation}
    In particular, we say $U$ is displacement monotone when $\lambda=0$, and displacement semi-monotone if it is displacement $\lambda$-monotone for some $\lambda>0$.
\end{definition}
\begin{remark}
    (i) When $U,\partial_x U\in C^1(\R^d\times \mathcal{P}_2(\R^d))$, the equivalent differential \\forms for Lasry-Lions monotonicity (\ref{LL monotone}) and displacement monotonicity (\ref{disp monotone}) with $\lambda=0$ are given in \cite[Section 2]{mou2022displacement}. They are
    \begin{equation}
        \tilde{\E}\big[\langle\partial_{x\mu}U(\xi,\mathcal{L}_{\xi},\tilde{\xi})\tilde{\eta},\eta\rangle\big]\geq 0,\quad\forall\xi,\eta\in\mathbb{L}^2(\mathcal{F}_T^1),
    \end{equation}
    and
    \begin{equation}
        \tilde{\E}\big[\langle\partial_{x\mu}U(\xi,\mathcal{L}_{\xi},\tilde{\xi})\tilde{\eta},\eta\rangle+\langle\partial_{xx}U(\xi,\mathcal{L}_{\xi})\eta,\eta\rangle\big]\geq 0,\quad\forall\xi,\eta\in\mathbb{L}^2(\mathcal{F}_T^1).
    \end{equation}
    For $\lambda\geq 0$, the displacement $\lambda$-monotonicity (\ref{disp monotone}) can be easily shown to be equivalent to
    \begin{equation}
        \tilde{\E}\big[\langle\partial_{x\mu}U(\xi,\mathcal{L}_{\xi},\tilde{\xi})\tilde{\eta},\eta\rangle+\langle\partial_{xx}U(\xi,\mathcal{L}_{\xi})\eta,\eta\rangle+\lambda|\eta|^2\big]\geq 0,\quad\forall\xi,\eta\in\mathbb{L}^2(\mathcal{F}_T^1).
    \end{equation}
    (ii) It is clear that the displacement monotonicity implies the displacement semi-monotonicity. Moreover, if  $\partial_{xx}U$ is bounded, Lasry-Lions monotonicity also implies displacement semi-monotonicity. 
\end{remark}

\par Throughout this paper, we fix a constant $\lambda\geq 0$, and establish the global well-posedness of the master equation under both Lasry-Lions monotonicity and displacement $\lambda$-monotonicity conditions, respectively. We collect the following assumptions on $G$ and $\widehat{H}$:
\begin{assumption}
\label{G regularity}
    $G\in\mathcal{C}^2(\R^d\times\mathcal{P}_2(\R^d))$ and there exist constants $L_x^G,L_{\mu}^G$ such that
    \[|\partial_x G|,|\partial_{xx}G|\leq L_{x}^G,\quad|\partial_{x\mu}G|\leq L_{\mu}^G.\]
\end{assumption}
\begin{remark}\label{G regularity remark}
    In the above assumption, when $G\in\mathcal{C}^2(\R^d\times\mathcal{P}_2(\R^d))$, the uniform boundedness $|\partial_{x\mu} G|\leq L_{\mu}^G$ is equivalent to the $W_1$-Lipschitz continuity of $\partial_x G$ in $\mu$, with the uniform $W_1$-Lipschitz constant $L_{\mu}^G$. This implies further the $W_2$-Lipschitz continuity of $\partial_x G$ in $\mu$, and we denote this uniform $W_2$-Lipschitz constant by $\tilde{L}_{\mu}^G\leq L_{\mu}^G$.
\end{remark}
\begin{assumption}
\label{H regularity}
    (i) $\widehat{H}\in\mathcal{C}^2(\R^d\times\R^d\times\mathcal{P}_2(\R^{2d}))$, and for any $R>0$, there exists $L^H(R)$ such that
    \[
    \begin{gathered}
        |\partial_x\widehat{H}|,|\partial_p\widehat{H}|,|\partial_{xp}\widehat{H}|,|\partial_{xx}\widehat{H}|,|\partial_{pp}\widehat{H}|\leq L^H(R)\quad\text{on }\mathbb{R}^d\times B_R\times\mathcal{P}_2(\mathbb{R}^{2d}),\\
        |\partial_{\rho}\widehat{H}|,|\partial_{x\rho}\widehat{H}|,|\partial_{p\rho}\widehat{H}|\leq L^H(R)\quad\text{on }\mathbb{R}^d\times B_R\times\mathcal{P}_2(\mathbb{R}^{2d})\times\mathbb{R}^{2d}.
    \end{gathered}
    \]

    (ii) There exists $C_0>0$ such that  
    \[
    |\partial_x\widehat{H}(x,p,\rho)|\leq C_0(1 +|p|)\quad\text{for any } (x,p,\rho)\in\R^d\times\R^d\times\mathcal{P}_2(\R^d).
    \]
    (iii) There exist constants $c_0>c_1>0$ such that $\partial_{pp}\widehat{H}\leq-c_0 I_d$ and $|\partial_{p\rho_2}\widehat{H}|\leq c_1$, where $I_d$ denotes the $d\times d$ identity matrix. We denote $C_1:=c_0-c_1>0$. Here,
    \[\partial_\rho\widehat{H}(x,p,\rho,\tilde{x},\tilde{p})=\big(\partial_{\rho_1}\widehat{H}(x,p,\rho,\tilde{x},\tilde{p}),\partial_{\rho_2}\widehat{H}(x,p,\rho,\tilde{x},\tilde{p})\big)\in\R^d\times\R^d.\]
\end{assumption}
\begin{remark}\label{remark H regularity}
(i) In the above Assumption \ref{H regularity}(i), we only require the local Lipschitz continuity of $\widehat{H}$ with respect to $p$. This means that we allow $\widehat{H}$ to have arbitrary order growth in $p$.\\
(ii) In a standard MFG, we only need to make assumptions on $\partial_{pp}H$ for the concavity. However, in our MFGC problem, $\partial_{p\rho_2}\widehat{H}$ contributes to the concavity of $\widehat{H}$ together with $\partial_{pp}\widehat{H}$. To see this, we compute for any $\xi,\eta^1,\eta^2\in\mathbb{L}^2(\mathcal{F}_T^1)$,
\[
\begin{aligned}
    &\E\Big[\big\langle\partial_p\widehat{H}(\xi,\eta^1,\mathcal{L}_{(\xi,\eta^1)})-\partial_p\widehat{H}(\xi,\eta^2,\mathcal{L}_{(\xi,\eta^2)}),\eta^1-\eta^2\big\rangle\Big]\\
    \leq& -c_0\E[|\eta^1-\eta^2|^2]+c_1|\E[\eta^1-\eta^2]||\tilde{\E}[\tilde{\eta}^1-\tilde{\eta}^2]|\\
    \leq& -C_1\E[|\eta^1-\eta^2|^2],
\end{aligned}
\]
where $c_0,c_1,C_1$ are defined in Assumption \ref{H regularity}(iii).
\end{remark}

For the reader's convenience, we conclude this section by presenting the following estimate, which relates to the a priori regularity estimate of $V$ in $x$.
\begin{proposition}[{\cite[Proposition 6.1]{mou2022displacement}}]
\label{Vxx bounded}
    Let Assumptions \ref{G regularity}, \ref{H regularity}(i)(ii) hold and $\rho:[0,T]\times \Omega\to\mathcal{P}_2(\R^{2d})$ be $\mathbb{F}^0$-progressively measurable with
    \[
    \sup_{t\in[0,T]}\E[\int_{\mathbb{R}^{2d}}|(x,p)|^2\,\rho_t(dx,dp)]<+\infty:
    \]
    (i) For any $x\in\mathbb{R}^d$ and for the $X^x$ in (\ref{decoupled FBSDE}), the following BSDE on $[t_0,T]$ has a unique solution with bounded $Z^x$:
	\begin{equation}
        Y_t^x=G(X_T^x,\mu_T)+\int_t^T\widehat{H}(X_s^x,Z_s^x,\rho_s)\,ds-\int_t^T Z_s^x\cdot dB_s-\int_t^T Z_s^{0,x}\cdot dB_s^0,
	\end{equation}
	where $\pi_1\#\rho=\mu$.\\
    (ii) Denote $u(t_0,x)=Y_{t_0}^x$. Then there exists $L_x^V>0$, depending only on $d,T$, $L_x^G$ in Assumption \ref{G regularity}, $L^H$ in Assumption \ref{H regularity}(i), and $C_0$ in Assumption \ref{H regularity}(ii), such that
    \[
    |\partial_x u(t_0,x)|,|\partial_{xx}u(t_0,x)|\leq L_x^V.
    \]
\end{proposition}

\section{Relation between forms of monotonicities for the Hamiltonian}
In this section, we introduce the integral forms of Lasry-Lions monotonicity and displacement $\lambda$-monotonicity conditions on the Hamiltonian. We first show that the differential form of Lasry-Lions monotonicity implies its integral form. We further show by a remark that the integral form is indeed a weaker requirement. Finally, the integral and differential forms of displacement $\lambda$-monotonicity condition on $\widehat{H}$ are shown to be equivalent.
\par We first give the Lasry-Lions monotonicity condition on $\widehat{H}$ in the integral form below.
\begin{assumption}
\label{assumption H LL}
    For any $\varphi(x,\mu)\in\mathcal{C}^1(\R^d\times\mathcal{P}_2(\R^d);\R^d)$ with bounded $\varphi,\partial_x\varphi$ and $\mu^1,\mu^2\in\mathcal{P}_2(\R^d)$,
    \begin{equation}
    \label{H LL int}
    \begin{aligned}
        &\int_{\R^d}\big(\widehat{H}(x,\varphi(x,\mu^1),\rho^1)-\widehat{H}(x,\varphi(x,\mu^2),\rho^2)\big)(\mu^1-\mu^2)(dx)\\
        &-\int_{\R^d}\big(\varphi(x,\mu^1)-\varphi(x,\mu^2)\big)\cdot\big(\partial_p\widehat{H}(x,\varphi(x,\mu^1),\rho^1)\mu^1(dx)\\
        &\qquad\qquad-\partial_p\widehat{H}(x,\varphi(x,\mu^2),\rho^2)\mu^2(dx)\big)\geq 0,
    \end{aligned}
    \end{equation} 
    where $\rho^i=\big(id,\varphi(\cdot,\mu^i)\big)\#\mu^i,\,i=1,2$.
\end{assumption}
Next we recall the following Lasry-Lions monotonicity condition on $\widehat{H}$ in the differential form from \cite{mou2022propagation}.

\begin{assumption}\label{assumption H LL diff}
    For any $\xi,\eta,\gamma,\zeta\in\mathbb{L}^2(\mathcal{F}_T^1)$ and Lipschitz continuous $\varphi:\R^d\rightarrow\R^d$,
    \begin{equation}
    \label{H LL diff}
    \begin{aligned}
        &\tilde{\E}\Big[\big\langle\zeta,\widehat{H}_{pp}(\xi)\zeta\big\rangle-\big\langle\eta,\widehat{H}_{x\rho_1}(\xi,\tilde{\xi})\tilde{\eta}+\widehat{H}_{x\rho_2}(\xi,\tilde{\xi})[\tilde{\gamma}+\tilde{\zeta}]\big\rangle\\
        &-\big\langle\gamma-\zeta,\widehat{H}_{p\rho_1}(\xi,\tilde{\xi})\tilde{\eta}+\widehat{H}_{p\rho_2}(\xi,\tilde{\xi})[\tilde{\gamma}+\tilde{\zeta}]\big\rangle\Big]\leq 0,\quad\text{where}\\
    \end{aligned}
    \end{equation}
    $\widehat{H}_{pp}(x):=\partial_{pp}\widehat{H}\big(x,\varphi(x),\mathcal{L}_{(\xi,\varphi(\xi))}\big)$, $\widehat{H}_{x\rho}
(x,\tilde{x}):=\partial_{x\rho}\widehat{H}\big(x,\varphi(x),\mathcal{L}_{(\xi,\varphi({\xi}))},\tilde{x},\varphi(\tilde{x})\big)$, and similarly for $\widehat{H}_{p\rho}(x,\tilde{x})$.
\end{assumption}
To show that (\ref{H LL diff}) implies (\ref{H LL int}), we need the following chain rule:
\begin{lemma}
\label{chain rule}
    Let $\{\mu_t\}_{t\in[0,1]}\subset\mathcal{P}_2(\R^d)$ be a continuous measure flow satisfying $\sup_{t\in[0,1]}\int_{\R^d}|x|^2\mu_t(dx)<\infty$. Let $v_t\in L^{\infty}([0,1]\times\R^d;\R^d)$ be a velocity field for $\mu_t$, that is, $v_t$ is the unique solution to the equation $\partial_t\mu_t+\divergence(\mu_t v_t)=0$ in the sense of distribution. Let $\varphi_t(x)\in\mathcal{C}_b^1([0,1]\times\R^d;\R^d)$. Define the measure flow $\rho_t:=\varphi_t(\cdot)\#\mu_t$. Then, for any function $f\in\mathcal{C}_b^1(\mathcal{P}_2(\R^d))$, we have the chain rule:
    \begin{equation}
        \frac{d}{dt}f(\rho_t)=\int_{\R^d}\partial_{\mu}f\big(\rho_t,\varphi_t(x)\big)\cdot\big[\partial_x\varphi_t(x)v_t(x)+\partial_t\varphi_t(x)\big]\mu_t(dx).
    \end{equation}
\end{lemma}

\begin{proof}
    Since $v_t$ satisfies the equation in the sense of distribution, for any $\Psi_t(x)\in\mathcal{C}_c^1([0,1]\times\R^d)$ and $[t_1,t_2]\subset[0,1]$ we have
    \begin{equation}
    \label{chain rule test}
        \int_{\R^d}\Psi_{t_2}(x)\mu_{t_2}(dx)-\int_{\R^d}\Psi_{t_1}(x)\mu_{t_1}(dx)=\int_{t_1}^{t_2}\int_{\R^d}\big[\partial_x\Psi_t(x)\cdot v_t(x)+\partial_t\Psi_t(x)\big]\mu_t(dx)\,dt.
    \end{equation}
    We claim that we can set the test functions to be $\Psi_t(x)\in\mathcal{C}_b^1([0,1]\times\R^d)$. In fact, we can approximate the function $\Psi_t$ by $\Psi_t\chi_R$ with $\chi_R\in\mathcal{C}_c^{\infty}(\R^d),\,0\leq\chi_R\leq 1,|\nabla\chi_R|\leq 2$ and $\chi_R=1$ on $B_R$; see \cite[Remark 8.1.1]{ambrosio2008gradient}. It is clear that (\ref{chain rule test}) holds if we replace $\Psi_t$ with $\Psi_t\chi_R$. To conclude the convergence as $R\rightarrow\infty$, we need to examine the integrals
    \[
    \begin{gathered}  \int_{\R^d}\Psi_t\chi_R\,\mu_t(dx),\quad\int_0^1\int_{\R^d}(\partial_x\Psi_t\chi_R)\cdot v_t\,\mu_t(dx)dt,\\\int_0^1\int_{\R^d}(\Psi_t\partial_x\chi_R)\cdot v_t\,\mu_t(dx)dt,\quad\int_0^1\int_{\R^d}\partial_t\Psi_t\chi_R\,\mu_t(dx)dt.
    \end{gathered}
    \]
    Since $\Psi_t(x)\in\mathcal{C}_b^1([0,1]\times\R^d)$, $\sup_{t\in[0,1]}\int_{\R^d}|x|^2\mu_t(dx)<\infty$ and $v_t\in\mathbb{L}^{\infty}([0,1]\times\R^d;\R^d)$, we can apply dominated convergence theorem to obtain the desired result.
    \par Since $f\in\mathcal{C}_b^1(\mathcal{P}_2(\R^d))$, the existence of linear functional derivative $\frac{\delta f}{\delta\mu}$ can be guaranteed by \cite[I, Proposition 5.51]{carmona2018probabilistic}, and moreover, $\partial_x\frac{\delta f}{\delta\mu}(\mu,x)=\partial_{\mu}f(\mu,x)$. Then we can compute that for any $\mu\in\mathcal{P}_2(\R^d)$, $\partial_x\frac{\delta f}{\delta\mu}(\mu,\varphi_t(x))=\partial_{\mu}f(\mu,\varphi_t(x))\partial_x\varphi_t(x)$ and $\partial_t\frac{\delta f}{\delta\mu}(\mu,\varphi_t(x))=\partial_{\mu}f(\mu,\varphi_t(x))\partial_t\varphi_t(x)$ are continuous and bounded, so $\frac{\delta f}{\delta\mu}(\mu,\varphi_t(x))$ is an admissible test function. 
    \par Next, for any $[t,t+\epsilon]\subset[0,1]$, we substitute the above test function into (\ref{chain rule test}) and obtain
    \[
    \begin{aligned}
    &\int_{\R^d}\frac{\delta f}{\delta\mu}(\mu,x)\big(\rho_{t+\epsilon}(dx)-\rho_t(dx)\big)\\
    =&\int_{\R^d}\frac{\delta f}{\delta\mu}(\mu,\varphi_{t+\epsilon}(x))\mu_{t+\epsilon}(dx)-\int_{\R^d}\frac{\delta f}{\delta\mu}(\mu,\varphi_t(x))\mu_t(dx)\\
    =&\int_t^{t+\epsilon}\int_{\R^d}\partial_{\mu}f(\mu,\varphi_s(x))\cdot\big[\partial_x\varphi_s(x)v_s(x)+\partial_t\varphi_s(x)\big]\mu_s(dx)\,ds,
    \end{aligned}
    \]
    which implies
    \begin{equation}\label{chain rule 1}
    \begin{aligned}
        &\lim_{\epsilon\rightarrow 0}\frac{1}{\epsilon}\int_{\R^d}\frac{\delta f}{\delta\mu}(\mu,x)\big(\rho_{t+\epsilon}(dx)-\rho_t(dx)\big)\\
        =&\int_{\R^d}\partial_{\mu}f(\mu,\varphi_t(x))\cdot\big[\partial_x\varphi_t(x)v_t(x)+\partial_t\varphi_t(x)\big]\mu_t(dx)
    \end{aligned}
    \end{equation}
    due to the boundedness of the integrand.
    \par By a first-order expansion of $f(\rho_t)$ (see \cite[I, Proposition 5.44]{carmona2018probabilistic}), we have
    \begin{equation}\label{chain rule 2}
        f(\rho_{t+\epsilon})-f(\rho_t)=\int_{\R^d}\frac{\delta f}{\delta\mu}(\rho_t,x)(\rho_{t+\epsilon}(dx)-\rho_t(dx))+o(W_2(\rho_t,\rho_{t+\epsilon})).
    \end{equation}
    To see $W_2(\rho_{t+\epsilon},\rho_t)=O(\epsilon)$, we refer to \cite[Theorem 8.3.1]{ambrosio2008gradient}, which gives the estimate
    \begin{equation}
        W_2(\mu_{t+\epsilon},\mu_t)\leq\int_t^{t+\epsilon}\Big(\int_{\R^d}|v_s(x)|^2\,\mu_s(dx)\Big)^{\frac12}\,ds=O(\epsilon).
    \end{equation}
    Since $\varphi_t(x)\in\mathcal{C}_b^1([0,1]\times\R^d;\R^d])$, combining with the above estimate, we have
    \begin{equation}
        W_2^2(\rho_{t+\epsilon},\rho_t)\leq\E[|\varphi_{t+\epsilon}(\xi_{t+\epsilon})-\varphi_t(\xi_t)|^2]\leq C\epsilon^2+CW_2^2(\mu_{t+\epsilon},\mu_t)=O(\epsilon^2),
    \end{equation}
    where we choose random variables $\xi_t,\,\xi_{t+\epsilon}$ such that $\mathcal{L}_{\xi_t}=\mu_t,\,\mathcal{L}_{\xi_{t+\epsilon}}=\mu_{t+\epsilon}$ and $W_2^2(\mu_{t+\epsilon},\mu_t)=\E[|\xi_{t+\epsilon}-\xi_t|^2]$. Therefore, combining (\ref{chain rule 1}) with (\ref{chain rule 2}), we conclude that
    \begin{equation}
        \frac{d}{dt}f(\rho_t)=\int_{\R^d}\partial_{\mu}f\big(\rho_t,\varphi_t(x)\big)\cdot\big[\partial_x\varphi_t(x)v_t(x)+\partial_t\varphi_t(x)\big]\mu_t(dx).
    \end{equation}
\end{proof}
The relation of Lasry-Lions monotonicity for the Hamiltonian is established below.
\begin{theorem}\label{equiH_LL}
    Suppose $\widehat{H}\in\mathcal{C}^2(\R^d\times\R^d\times\mathcal{P}_2(\R^{2d}))$. If $\widehat{H}$ satisfies the differential form of Lasry-Lions monotonicity condition Assumption \ref{assumption H LL diff}, then $\widehat{H}$ satisfies the integral form of Lasry-Lions monotonicity condition Assumption \ref{assumption H LL}.
\end{theorem}
\begin{proof}     
We first use a $W_1$-geodesic interpolation argument as in \cite[Remark 2.4]{mou2022displacement}. By a standard density argument, we can assume that $\mu^i$ admit smooth densities $m^i\in\mathcal{C}^{\infty}(B_R)$ with $\min_{B_R}m^i>0$ for $i=1,2$. Define the interpolation $m_t=tm^1+(1-t)m^2$ for $t\in[0,1]$, and let $\mu_t$ denote the corresponding probability measure. Since $m_t$ is bounded away from $0$ on $[0,1]\times B_R$, then for each $t$, there exists a unique solution $\phi_t\in H_0^1(B_R^o)\cap \mathcal{C}^{\infty}(B_R)$ to the elliptic equation $\partial_t m_t+\divergence(m_t \nabla\phi_t)=0$. Let $v_t:=\nabla\phi_t$, then $v_t$ serves as a velocity for the curve $t\rightarrow\mu_t$. 
\par For the given $\varphi\in\mathcal{C}^1(\R^d\times\mathcal{P}_2(\R^d);\R^d)$ with bounded $\varphi,\partial_x\varphi$, and $\mu^1,\mu^2\in\mathcal{P}_2(\R^d)$, define $\varphi^i(\cdot):=\varphi(\cdot,\mu^i),\,\,\varphi_t(\cdot):=t\varphi^1(\cdot)+(1-t)\varphi^2(\cdot)$ and $\rho^i:=(id,\varphi(\cdot,\mu^i))\#\mu^i,\, \rho_t:=\big(id,\varphi_t(\cdot)\big)\#\mu_t$ for $i=1,2$. Noting that $\varphi_t(x)\in\mathcal{C}_b^1([0,1]\times\R^d)$, we can apply the chain rule Lemma \ref{chain rule} to compute that
\small\begin{align*}
    &\int_{\R^d}\big(\widehat{H}(x,\varphi^1(x),\rho^1)-\widehat{H}(x,\varphi^2(x),\rho^2)\big)(\mu^1-\mu^2)(dx)\\
    &-\int_{\R^d}\big(\varphi^1(x)-\varphi^2(x)\big)\cdot\big(\partial_p\widehat{H}(x,\varphi^1(x),\rho^1)\mu^1(dx)-\partial_p\widehat{H}(x,\varphi^2(x),\rho^2)\mu^2(dx)\big)\\
    =&-\int_0^1\int_{\R^d}\frac{d}{dt}\widehat{H}\big(x,\varphi_t(x),\rho_t\big)\nabla(m_tv_t)\,dx\,dt\\
    &-\int_0^1\int_{\R^d}\frac{d}{dt}\Big[\big(\varphi^1(x)-\varphi^2(x)\big)\cdot\partial_p\widehat{H}\big(x,\varphi_t(x),\rho_t\big)m_t(x)\Big]\,dx\,dt\\
    =&\int_0^1\int_{\R^d}\int_{\R^d}\Big[\Big\langle v_t(x),\widehat{H}_{x\rho_1}(x,\tilde{x})v_t(\tilde{x})+\widehat{H}_{x\rho_2}(x,\tilde{x})\big(\partial_x\varphi_t(\tilde{x})v_t(\tilde{x})+\varphi^1(\tilde{x})-\varphi^2(\tilde{x})\big)\Big\rangle\\
     &\quad+\Big\langle\partial_x\varphi_t(x)v_t(x)-\varphi^1(x)+\varphi^2(x),\widehat{H}_{p\rho_1}(x,\tilde{x})v_t(\tilde{x})+\widehat{H}_{p\rho_2}(x,\tilde{x})\\
     &\quad\qquad\cdot\big(\partial_x\varphi_t(\tilde{x})v_t(\tilde{x})+\varphi^1(\tilde{x})-\varphi^2(\tilde{x})\big)\Big\rangle\\
    &\quad-\Big\langle\varphi^1(x)-\varphi^2(x),\widehat{H}_{pp}(x)\big(\varphi^1(x)-\varphi^2(x)\big)\Big\rangle\Big]m_t(x)m_t(\tilde{x})\,dx\,d\tilde{x}\,dt,
\end{align*}\normalsize
where in the last equality we use integration by parts. The notations for derivatives are defined as $\widehat{H}_{pp}(x):=\partial_{pp}H(x,\varphi_t(x),\rho_t)$, $\widehat{H}_{x\rho}(x,\tilde{x}):=\partial_{x\rho}\widehat{H}(x,\varphi_t(x),\rho_t,\tilde{x},\varphi_t(\tilde{x}))$, and it is similar for $\widehat{H}_{p\rho}(x,\tilde{x})$. From (\ref{H LL diff}), we see that the above integral is nonnegative so $\widehat{H}$ satisfies Assumption \ref{assumption H LL}.
\end{proof}
\begin{remark}
    Conversely, we can see that the integral form of Lasry-Lions monotonicity implies its differential form for specific function $\varphi$ and random variables $\zeta,\gamma$. This is why we mentioned the integral form for Lasry-Lions monotonicity condition is weaker than the differential form. For any $\varphi\in\mathcal{C}^1(\R^d\times\mathcal{P}_2(\R^d);\R^d)$ with bounded $\varphi,\partial_x\varphi$, and $\mu^1,\mu^2\in\mathcal{P}_2(\R^d)$ with $\xi^i\in\mathbb{L}^2(\mathcal{F}_T^1;\mu^i)$, we denote $\varphi^i(\cdot):=\varphi(\cdot,\mu^i),\,\rho^i:=\big(id,\varphi(\cdot,\mu^i)\big)\#\mu^i$ for $i=1,2$. For $\lambda\in[0,1]$, we denote $\xi_{\lambda}:=\lambda\xi^1+(1-\lambda)\xi^2,\,\varphi_{\lambda}(\cdot):=\lambda\varphi^1(\cdot)+(1-\lambda)\varphi^2(\cdot)$, $\rho_\lambda:=\mathcal{L}\big(\xi_{\lambda},\lambda \varphi^1(\xi^1)+(1-\lambda)\varphi^2(\xi^2)\big),\,\mu_{\lambda}(dx):=\lambda\mu^1(dx)+(1-\lambda)\mu^2(dx)$. Then \eqref{H LL int} implies
\begin{align*}
    0\leq&\int_{\R^d}\big(\widehat{H}(x,\varphi^1(x),\rho^1)-\widehat{H}(x,\varphi^2(x),\rho^2)\big)\,(\mu^1-\mu^2)(dx)\\
    &-\int_{\R^d}\big(\varphi^1(x)-\varphi^2(x)\big)\cdot\big[\partial_p\widehat{H}(x,\varphi^1(x),\rho^1)\,\mu^1(dx)-\partial_p\widehat{H}(x,\varphi^2(x),\rho^2)\,\mu^2(dx)\big]\\
    =&\tilde{\E}\int_0^1\int_0^1\Big\{\Big\langle\xi^1-\xi^2,\partial_{x\rho_1}\widehat{H}\big(\xi_{\theta},\varphi_{\lambda}(\xi_{\theta}),\rho_{\lambda},\tilde{\xi}_{\lambda},\lambda\varphi^1(\tilde{\xi}^1)+(1-\lambda)\varphi^2(\tilde{\xi}^2)\big)(\tilde{\xi}^1-\tilde{\xi}^2)\Big\rangle\\
    &\quad+\Big\langle\partial_x \varphi_{\lambda}(\xi_{\theta})(\xi^1-\xi^2),\partial_{p\rho_1}\widehat{H}(\tilde{\xi}^1-\tilde{\xi}^2)\Big\rangle+\Big\langle\xi^1-\xi^2,\partial_{x\rho_2}\widehat{H}\big(\varphi^1(\tilde{\xi}^1)-\varphi^2(\tilde{\xi}^2)\big)\Big\rangle\\
    &\quad+\Big\langle\partial_x\varphi_{\lambda}(\xi_{\theta})(\xi^1-\xi^2),\partial_{p\rho_2}\widehat{H}\big(\varphi^1(\tilde{\xi}^1)-\varphi^2(\tilde{\xi}^2)\big)\Big\rangle\,\Big\}d\theta\,d\lambda\\
    &-{\tilde{\E}}\int_0^1{\int_{\R^d}\Big\{}\Big\langle\varphi^1(x)-\varphi^2(x),\partial_{pp}\widehat{H}\big(x,\varphi_{\lambda}(x),\rho_{\lambda}\big)\big(\varphi^1(x)-\varphi^2(x)\big)\Big\rangle\\
    &\quad+\Big\langle\varphi^1(x)-\varphi^2(x),\partial_{p\rho_1}\widehat{H}(\tilde{\xi}^1-\tilde{\xi}^2)\Big\rangle\\
    &\quad+\Big\langle\varphi^1(x)-\varphi^2(x),\partial_{p\rho_2}\widehat{H}\big(\varphi(\tilde{\xi}^1,\mu^1)-\varphi(\tilde{\xi}^2,\mu^2)\big)\Big\rangle\,\Big\}\,\mu_{\lambda}(dx)d\lambda.
\end{align*}
\par For given $\xi,\eta\in\mathbb{L}^2(\mathcal{F}_T^1)$ with $\mathcal{L}_{\xi}=\mu\in\mathcal{P}_2(\R^d)$, we set $\xi^2=\xi,\,\xi^1=\xi+\epsilon\eta$, divide the above inequality by $\epsilon^2$ and let $\epsilon\rightarrow 0$. Note that
\[
    \partial_x\varphi_{\lambda}(\xi_{\theta})\rightarrow\partial_x\varphi(\xi),\quad\frac{1}{\epsilon}[\varphi(x,\mathcal{L}_{\xi+\epsilon\eta})-\varphi(x,\mathcal{L}_{\xi})]\rightarrow\tilde{\E}_{\mathcal{F}_T^1}[\partial_{\mu}\varphi(x,\mathcal{L}_{\xi},\tilde{\xi})\tilde{\eta}]
\]
as $\epsilon\rightarrow 0$. We conclude the Lasry-Lions monotonicity in the differential form (\ref{H LL diff}) for $\varphi\in\mathcal{C}^1(\R^d\times\mathcal{P}_2(\R^d);\R^d)$ and $\zeta:=\tilde{\E}_{\mathcal{F}_T^1}[\partial_{\mu}\varphi(\xi,\mu,\tilde{\xi})\tilde{\eta}],\,\gamma:=\partial_x\varphi(x,\mu)\eta$.
\end{remark}
\begin{remark}
    In \cite{kobe2022}, the authors proved the uniqueness of the MFGC system (\ref{FBSPDE}) through both the convexity and the following Lasry-Lions monotonicity condition for Lagrangian:
    \begin{equation}\label{f LL} \int_{\R^{2d}}\big(f(x,a,\rho^1)-f(x,a,\rho^2)\big)(\rho^1-\rho^2)(dx,da)\geq 0.
    \end{equation}
    In contrast, our Lasry-Lions monotonicity condition (\ref{H LL int}) for the Hamiltonian seems to be weaker if we translate our condition (\ref{H LL int}) back to the Lagrangian. More precisely, following the notations in Assumption \ref{assumption H LL}, we further denote $a^i:=\partial_p\widehat{H}(x,\varphi(x,\mu^i),\rho^i)$ for $i=1,2$, ad then we can rewrite (\ref{H LL int}) as
    \begin{equation}
        \begin{aligned}
        &\int_{\R^d}\big(\widehat{H}(x,\varphi(x,\mu^1),\rho^1)-\widehat{H}(x,\varphi(x,\mu^2),\rho^2)\big)(\mu^1-\mu^2)(dx)\\
        &-\int_{\R^d}\big(\varphi(x,\mu^1)-\varphi(x,\mu^2)\big)\cdot\big(\partial_p\widehat{H}(x,\varphi(x,\mu^1),\rho^1)\mu^1(dx)\\
        &\qquad\qquad-\partial_p\widehat{H}(x,\varphi(x,\mu^2),\rho^2)\mu^2(dx)\big)\\
        =&\int_{\R^{2d}}\big(f(x,a,\tilde{\rho}^1)-f(x,a,\tilde{\rho}^2)\big)(\tilde{\rho}^1-\tilde{\rho}^2)(dx,da)\\
        &+\int_{\R^d}\big(f(x,a^1,\tilde{\rho}^2)-f(x,a^2,\tilde{\rho}^2)-(a^1-a^2)\cdot\partial_a f(x,a^2,\tilde{\rho}^2)\big)\mu^1(dx)\\
        &+\int_{\R^d}\big(f(x,a^2,\tilde{\rho}^1)-f(x,a^1,\tilde{\rho}^1)-(a^2-a^1)\cdot\partial_a f(x,a^1,\tilde{\rho}^1)\big)\mu^2(dx),
    \end{aligned}
    \end{equation}
    where $\tilde{\rho}^i:=\Phi(\rho^i),i=1,2$ are defined in Assumption \ref{fixed point}. We see that, if (\ref{f LL}) and the convexity of Lagrangian hold, our Lasry-Lions monotonicity condition (\ref{H LL int}) holds.
\end{remark}
\par Next we focus on the displacement $\lambda$-monotonicity. We first give the following integral form of displacement $\lambda$-monotonicity condition on $\widehat{H}$.
\begin{assumption}
\label{assumption H disp}
    For any $\xi^i,\eta^i\in\mathbb{L}^2(\mathcal{F}_T^1)$ with $\mathcal{L}_{(\xi^i,\eta^i)}=\rho^i\in\mathcal{P}_2(\R^{2d}),\,i=1,2$,
    \begin{equation}
    \label{H displacement int}
    \begin{aligned}
        &\E\Big[\big\langle\partial_x\widehat{H}(\xi^1,\eta^1,\rho^1)-\partial_x\widehat{H}(\xi^2,\eta^2,\rho^2),\xi^1-\xi^2\big\rangle\Big]\\
        &-\E\Big[\big\langle\partial_p\widehat{H}(\xi^1,\eta^1,\rho^1)-\partial_p\widehat{H}(\xi^2,\eta^2,\rho^2),\eta^1-\eta^2\big\rangle\Big]\\
        &-2\lambda\E\Big[\big\langle\partial_p\widehat{H}(\xi^1,\eta^1,\rho^1)-\partial_p\widehat{H}(\xi^2,\eta^2,\rho^2),\xi^1-\xi^2\big\rangle\Big]\geq 0.
    \end{aligned}
    \end{equation}
\end{assumption}
The following differential form of displacement $\lambda$-monotonicity on $\widehat{H}$ is from \cite{mou2022propagation}.
\begin{assumption}
    For any $\xi,\eta,\gamma,\zeta\in\mathbb{L}^2(\mathcal{F}_T^1)$,
    \begin{equation}
    \label{H displacement diff}
    \begin{aligned}
        &\tilde{\E}\Big[\big\langle\zeta,\widehat{H}_{pp}(\xi,\eta)\zeta\big\rangle-\big\langle\gamma,[\widehat{H}_{xx}(\xi,\eta)-2\lambda\widehat{H}_{xp}(\xi,\eta)]\gamma\big\rangle\\
        &+\big\langle\zeta,[\widehat{H}_{p\rho_1}(\xi,\eta,\tilde{\xi},\tilde{\eta})-\widehat{H}_{x \rho_2}(\tilde{\xi},\tilde{\eta},\xi,\eta)+2\lambda\widehat{H}_{p\rho_2 }(\tilde{\xi},\tilde{\eta},\xi,\eta)]\tilde{\gamma}+2\lambda\widehat{H}_{pp}(\xi,\eta)\gamma\big\rangle\\
        &+\big\langle\zeta,\widehat{H}_{p\rho_2}(\xi,\eta,\tilde{\xi},\tilde{\eta})\tilde{\zeta}\big\rangle-\big\langle\gamma,[\widehat{H}_{x\rho_1}(\xi,\eta,\tilde{\xi},\tilde{\eta})-2\lambda\widehat{H}_{p\rho_1}(\xi,\eta,\tilde{\xi},\tilde{\eta})]\tilde{\gamma}\big\rangle\Big]\leq 0,
    \end{aligned}
    \end{equation}
    where $\widehat{H}_{pp}(x,p):=\partial_{pp}\widehat{H}(x,p,\mathcal{L}_{(\xi,\eta)})$, $\widehat{H}_{x\rho}(x,p,\tilde{x},\tilde{p}):=\partial_{x\rho}\widehat{H}(x,p,\mathcal{L}_{(\xi,\eta)},\tilde{x},\tilde{p})$, and it is similar for $\widehat{H}_{xp},\widehat{H}_{p\rho}$.
\end{assumption}
\par The integral and differential forms of displacement $\lambda$-monotonicity for $\widehat{H}$ can be shown to be equivalent.
\begin{theorem}\label{equiH_disp}
    Suppose $\widehat{H}\in\mathcal{C}^2(\R^d\times\R^d\times\mathcal{P}_2(\R^{2d}))$. Then the displacement $\lambda$-monotonicity of $\widehat{H}$ in the integral form (\ref{H displacement int}) is equivalent to its differential form (\ref{H displacement diff}).
\end{theorem}
\begin{proof}
   Note that
    \[
    \begin{aligned}
        &\partial_x\widehat{H}(\xi^1,\eta^1,\rho^1)-\partial_x\widehat{H}(\xi^2,\eta^2,\rho^2)\\
        =&\int_0^1\Big\{\partial_{xx}\widehat{H}(\xi_{\theta},\eta_{\theta},\mathcal{L}_{(\xi_{\theta},\eta_{\theta})})(\xi^1-\xi^2)+\partial_{xp}\widehat{H}(\xi_{\theta},\eta_{\theta},\mathcal{L}_{(\xi_{\theta},\eta_{\theta})})(\eta^1-\eta^2)\\
        &+\tilde{\E}\Big[\partial_{x\rho_1}\widehat{H}(\xi_{\theta},\eta_{\theta},\mathcal{L}_{(\xi_{\theta},\eta_{\theta})},\tilde{\xi}_{\theta},\tilde{\eta}_{\theta})(\tilde{\xi}^1-\tilde{\xi}^2)\\
        &\qquad+\partial_{x\rho_2}\widehat{H}(\xi_{\theta},\eta_{\theta},\mathcal{L}_{(\xi_{\theta},\eta_{\theta})},\tilde{\xi}_{\theta},\tilde{\eta}_{\theta})(\tilde{\eta}^1-\tilde{\eta}^2)\Big]\Big\}\,d\theta
    \end{aligned}
    \]
    and
    \[
    \begin{aligned}
        &\partial_p\widehat{H}(\xi^1,\eta^1,\rho^1)-\partial_p\widehat{H}(\xi^2,\eta^2,\rho^2)\\
        =&\int_0^1\Big\{\partial_{xp}\widehat{H}(\xi_{\theta},\eta_{\theta},\mathcal{L}_{(\xi_{\theta},\eta_{\theta})})(\xi^1-\xi^2)+\partial_{pp}\widehat{H}(\xi_{\theta},\eta_{\theta},\mathcal{L}_{(\xi_{\theta},\eta_{\theta})})(\eta^1-\eta^2)\\
        &+\tilde{\E}\Big[\partial_{p\rho_1}\widehat{H}(\xi_{\theta},\eta_{\theta},\mathcal{L}_{(\xi_{\theta},\eta_{\theta})},\tilde{\xi}_{\theta},\tilde{\eta}_{\theta})(\tilde{\xi}^1-\tilde{\xi}^2)\\
        &\qquad+\partial_{p\rho_2}\widehat{H}(\xi_{\theta},\eta_{\theta},\mathcal{L}_{(\xi_{\theta},\eta_{\theta})},\tilde{\xi}_{\theta},\tilde{\eta}_{\theta})(\tilde{\eta}^1-\tilde{\eta}^2)\Big]\Big\}\,d\theta.
    \end{aligned}
    \]
    Then we immediately see that (\ref{H displacement diff}) implies (\ref{H displacement int}) by choosing $\xi=\xi_{\theta},\,\eta=\eta_{\theta},\,\gamma=\xi^1-\xi^2,\,\zeta=\eta^1-\eta^2$. Conversely, setting $\xi^2=\xi,\,\xi^1=\xi+\epsilon\gamma,\,\eta^2=\eta,\,\eta^1=\eta+\epsilon\zeta$ and let $\epsilon\rightarrow 0$, we see that (\ref{H displacement int}) implies (\ref{H displacement diff}).
\end{proof}
\begin{remark}
    When the model does not involve a mean field term for the control, the problem reduces to a classical MFG. To compare with this degenerate case, we let $\widehat{H}_{x\rho_2}=\widehat{H}_{p\rho_2}=0$ in the monotonicity conditions (\ref{H LL diff}) and (\ref{H displacement diff}). 
    \par By completing the square for the $\widehat{H}_{pp}$ term, we observe that the displacement $\lambda$-monotonicity condition (\ref{H displacement diff}) is exactly the displacement monotonicity for the Hamiltonian defined in \cite{mou2022displacement} when $\lambda=0$.
    \par For the Lasry-Lions monotonicity, it is meaningless to discuss the degenerate case for a nonseparable Hamiltonian. Specifically, if we set $\widehat{H}_{x\rho_2}=\widehat{H}_{p\rho_2}=0$ in (\ref{H LL diff}), only the first-order term in $\gamma$ remains. Hence the inequality cannot hold unless $\widehat{H}_{p\rho_1}=0$, which is exactly the separable case. 
\end{remark}

\section{Propagation of the monotonicities}
The first step to global well-posedness is to show that either Lasry-Lions monotonicity or displacement $\lambda$-monotonicity is propagated along any classical solution $V$ to the master equation (\ref{master equation}). In this section, we shall establish the propagation of monotonicity conditions in their integral forms.
\par In \cite{mou2022propagation} the authors have proved the propagation of monotonicity conditions in their differential forms for the MFGC problem. However, in \cite{mou2022propagation}, and \cite{mou2022displacement} which studies standard MFG, it is assumed that $V$ is a classical solution to the master equation (\ref{master equation}) with further regularities $\partial_{xx}V(t,\cdot,\cdot)\in\mathcal{C}^2(\R^d\times\mathcal{P}_2(\R^d))$, $\partial_{x\mu}V(t,\cdot,\cdot,\cdot)\in\mathcal{C}^2(\R^d\times\mathcal{P}_2(\R^d)\times\R^d)$, and all of the second and higher-order derivatives involved above are uniformly bounded and continuous in $t$. With the new expressions of monotonicities in their integral forms, we can avoid differentiating the master equation so that the assumptions on higher-order derivatives may be dropped. In the later discussion, we can see that the proof of the Lipschitz continuity in measure also benefits from the integral forms. Therefore, unlike in \cite{mou2022displacement,mou2022propagation}, here only the $\mathcal{C}^2$-regularities for data are needed for the global well-posedness. 
\par The propagation of Lasry-Lions monotonicity in the integral form is given below.
\begin{theorem}\label{propagation LL}
    Suppose $V$ is a classical solution to the master equation (\ref{master equation}) with bounded $\partial_x V,\partial_{xx}V$ and $\partial_{x\mu}V$. Let Assumptions \ref{fixed point}, \ref{H regularity}(i) and \ref{assumption H LL} hold. If moreover $G$ satisfies the Lasry-Lions monotonicity condition (\ref{LL monotone}), then $V(t,\cdot,\cdot)$ also satisfies the Lasry-Lions monotonicity condition (\ref{LL monotone}) for all $t\in[t_0,T]$.
\end{theorem}
\begin{proof}
    For any $\xi^i\in\mathbb{L}^2(\mathcal{F}_{t_0}),\,i=1,2$, let $X_t^i$ be the strong solution to the SDE
    \begin{equation}
        X_t^i=\xi^i+\int_{t_0}^t\partial_p\widehat{H}\big(X_s^i,\partial_x V(s,X_s^i,\mu_s^i),\rho_s^i\big)\,ds+B_t^{t_0}+\beta B_t^{0,t_0},
    \end{equation}
    where $\mu_s^i:=\mathcal{L}_{X_s^i|\mathcal{F}_s^0}$ and $\rho_s^i:=\mathcal{L}_{(X_s^i,\partial_x V(s,X_s^i,\mu_s^i))|\mathcal{F}_s^0}$. Define $u^i(t,x):=V(t,x,\mu_t^i)$ and $v^i(t,x):=\beta\E_{\mathcal{F}_t^0}[\partial_{\mu}V(t,x,\mu_t^i,\tilde{X}_t^i)]$, where $\tilde{X}^i$ is an independent copy of $X^i$ conditioning on $\mathbb{F}^0$. Then $(\mu^i,u^i,v^i)$ is a classical solution to the MFGC system (\ref{FBSPDE}) with initial condition $\mu^i_{t_0}=\mathcal{L}_{\xi^i}$.
    \par Using the equation satisfied by $u^i$, by the It\^{o}-Wentzell formula, we compute that for $i,j=1,2$,
    \[
    \begin{aligned}
        du^i(t,X_t^j)=&\big[\partial_x u^i(t,X_t^j)\cdot\partial_p\widehat{H}\big(X_t^j,\partial_x u^j(t,X_t^j),\rho_t^j\big)-\widehat{H}\big(X_t^j,\partial_x u^i(t,X_t^j),\rho_t^i\big)\big]\,dt\\
        &+\partial_x u^i(t,X_t^j)\cdot d B_t+\big(v^i(t,X_t^j)+\beta\partial_x u^i(t,X_t^j)\big)\cdot dB_t^0.
    \end{aligned}
    \]
    Hence, by the Lasry-Lions monotonicity conditions (\ref{LL monotone}) for $G$ and (\ref{H LL int}) for $\widehat{H}$, we obtain
    \[
    \begin{aligned}
        &\E[u^1(t,X_t^1)+u^2(t,X_t^2)-u^1(t,X_t^2)-u^2(t,X_t^1)]\\
        =&\int_{\R^d}\big(G(x,\mu_T^1)-G(x,\mu_T^2)\big)\,(\mu_T^1-\mu_T^2)(dx)\\
        &+\int_t^T\Big[\int_{\R^d}\big(\widehat{H}(x,\partial_x u^1(s,x),\rho_s^1)-\widehat{H}(x,\partial_x u^2(s,x),\rho_s^2)\big)\,(\mu_s^1-\mu_s^2)(dx)\\
        &-\int_{\R^d}\big(\partial_x u^1(s,x)-\partial_x u^2(s,x)\big)\\
        &\qquad\cdot\big[\partial_p\widehat{H}(x,\partial_x u^1(s,x),\rho_s^1)\,\mu_s^1(dx)-\partial_p\widehat{H}(x,\partial_x u^2(s,x),\rho_s^2)\,\mu_s^2(dx)\big]\Big]\,ds\\
        \geq &0,
    \end{aligned}
    \]
    which is exactly the Lasry-Lions monotonicity condition for $V(t,\cdot,\cdot)$.
\end{proof}
\par We next introduce the propagation of displacement $\lambda$-monotonicity in the integral form.
\begin{theorem}\label{propagation disp}
    Suppose that $V$ is a classical solution to the master equation (\ref{master equation}) with bounded $\partial_x V,\partial_{xx}V$ and $\partial_{x\mu}V$. Let Assumptions \ref{fixed point}, \ref{H regularity}(i) and \ref{assumption H disp} hold. If moreover $G$ satisfies the displacement $\lambda$-monotonicity condition (\ref{disp monotone}), then $V(t,\cdot,\cdot)$ also satisfies the displacement $\lambda$-monotonicity condition (\ref{disp monotone}) for all $t\in[t_0,T]$.
\end{theorem}
\begin{proof}
    We follow the notations in Theorem \ref{propagation LL}. In addition, we define $Y_t^i:=\partial_x u^i(t,X_t^i),\,Z_t^i:=\partial_{xx}u^i(t,X_t^i),\,Z_t^{0,i}:=\beta\partial_{xx}u^i(t,X_t^i)+\partial_x v^i(t,X_t^i)$. Using the MFGC system (\ref{FBSPDE}), we can verify that $(X^i,Y^i,Z^i,Z^{0,i})$ is a strong solution to the FBSDE:
    \[
    \left\{\begin{aligned}
        &X_t^i=\xi^i+\int_{t_0}^t\partial_p\widehat{H}\big(X_s^i,Y_s^i,\rho_s^i\big)\,ds+B_t^{t_0}+\beta B_t^{0,t_0},\\
        &Y_t^i=\partial_x G(X_T^i,\mu_T^i)+\int_t^T\partial_x\widehat{H}(X_s^i,Y_s^i,\rho_s^i)\,ds-\int_t^T Z_s^i\cdot dB_s-\int_t^T Z_s^{0,i}\cdot dB_s^0.
    \end{aligned}\right.
    \]
    Note that we cannot directly apply It\^{o}-Wentzell's formula to $\partial_x u^i$ due to the lack of regularities, so a standard mollification argument is needed here. We refer to \cite[Proposition 3.3]{bansil2025degenerate} for details.
    \par Now applying It\^{o}'s formula and taking the expectation, we obtain that
    \[
    \begin{aligned}
        &\E\Big[\big\langle\partial_x V(t,X_t^1,\mu_t^1)-\partial_x V(t,X_t^2,\mu_t^2),X_t^1-X_t^2\big\rangle+\lambda|X_t^1-X_t^2|^2\Big]\\
        =&\E\Big[\big\langle Y_t^1-Y_t^2,X_t^1-X_t^2\big\rangle+\lambda|X_t^1-X_t^2|^2\Big]\\
        =&\E\Big[\big\langle\partial_x G(X_T^1,\mu_T^1)-\partial_x G(X_T^2,\mu_T^2),X_T^1-X_T^2\big\rangle+\lambda|X_T^1-X_T^2|^2\Big]\\
        &+\int_t^T\E\Big[\big\langle\partial_x\widehat{H}(X_s^1,Y_s^1,\rho_s^1)-\partial_x\widehat{H}(X_s^2,Y_s^2,\rho_s^2),X_s^1-X_s^2\big\rangle\\
        &\qquad-\big\langle\partial_p\widehat{H}(X_s^1,Y_s^1,\rho_s^1)-\partial_p\widehat{H}(X_s^2,Y_s^2,\rho_s^2),Y_s^1-Y_s^2\big\rangle\\
        &\qquad-2\lambda\big\langle\partial_p\widehat{H}(X_s^1,Y_s^1,\rho_s^1)-\partial_p\widehat{H}(X_s^2,Y_s^2,\rho_s^2),X_s^1-X_s^2\big\rangle\Big]\,ds\\
        \geq&0,
    \end{aligned}
    \]
    where we use the displacement $\lambda$-monotonicity condition (\ref{disp monotone}) for $G$ and (\ref{H displacement int}) for $\widehat{H}$. This is exactly the displacement $\lambda$-monotonicity condition for $V(t,\cdot,\cdot)$.
\end{proof}

\section{The uniform Lipschitz continuity in the measure variable}
In this section we will show that, if a classical solution $V$ to the master equation satisfies displacement $\lambda$-monotonicity condition, then $\partial_x V$ is always uniformly $W_2$-Lipschitz continuous in the measure variable. Moreover, with the similar method we can show that if $V$ satisfies Lasry-Lions monotonicity condition, then $\partial_x V$ is uniformly $W_1$-Lipschitz continuous. Using the monotonicities in their integral forms allows us to bypass differentiating the master equation to obtain the Lipschitz estimate, thereby requiring less regularity assumptions on data.
\subsection{The \texorpdfstring{$W_2$}{W2}-Lipschitz continuity under displacement \texorpdfstring{$\lambda$}{lambda}-monotonicity condition}
We first show that the displacement $\lambda$-monotonicity of $V(t,\cdot,\cdot)$ implies the uniform Lipschitz continuity of $\partial_x V$ in the measure variable under $W_2$.
\begin{theorem}
\label{W2 displacement}
    Suppose $V$ is a classical solution to the master equation (\ref{master equation}) with bounded $\partial_x V,\partial_{xx}V$ and $\partial_{x\mu}V$. Let Assumptions \ref{fixed point},  \ref{G regularity} and \ref{H regularity}(i)(iii) hold. Furthermore, suppose that $V(t,\cdot,\cdot)$ satisfies the displacement $\lambda$-monotonicity condition (\ref{disp monotone}). Then $\partial_x V$ is uniformly Lipschitz continuous with respect to $\mu$ under $W_2$, with the Lipschitz constant depending only on $d,T,\lambda,\|\partial_x V\|_{L^{\infty}},\|\partial_{xx}V\|_{L^{\infty}}$ and $L_x^G,L_{\mu}^G$ in Assumption \ref{G regularity}, $L^H,C_1$ in Assumption \ref{H regularity}.
\end{theorem}
\begin{proof}
    Recall the notations in Theorem \ref{propagation LL} and Theorem \ref{propagation disp}. We denote $\partial_x V_t^1:=\partial_x V(t,X_t^1,\mu_t^1),\partial_x V_t^2:=\partial_x V(t,X_t^2,\mu_t^2)$.  Following the computations in Theorem \ref{propagation disp}, we have
    \[
    \begin{aligned}
        &\E\Big[\big\langle\partial_x V(t,X_t^1,\mu_t^1)-\partial_x V(t,X_t^2,\mu_t^2),X_t^1-X_t^2\big\rangle+\lambda|X_t^1-X_t^2|^2\Big]\\
        =&\E\Big[\big\langle\partial_x V(t_0,\xi^1,\mu_{t_0}^1)-\partial_x V(t_0,\xi^2,\mu_{t_0}^2),\xi^1-\xi^2\big\rangle+\lambda|\xi^1-\xi^2|^2\Big]\\
        &-\int_{t_0}^t\E\Big[\big\langle\partial_x\widehat{H}(X_s^1,\partial_x V_s^1,\rho_s^1)-\partial_x\widehat{H}(X_s^2,\partial_x V_s^2,\rho_s^2),X_s^1-X_s^2\big\rangle\\
        &\qquad-\big\langle\partial_p\widehat{H}(X_s^1,\partial_x V_s^1,\rho_s^1)-\partial_p\widehat{H}(X_s^2,\partial_x V_s^2,\rho_s^2),\partial_x V_s^1-\partial_x V_s^2\big\rangle\\
        &\qquad-2\lambda\big\langle\partial_p\widehat{H}(X_s^1,\partial_x V_s^1,\rho_s^1)-\partial_p\widehat{H}(X_s^2,\partial_x V_s^2,\rho_s^2),X_s^1-X_s^2\big\rangle\Big]\,ds\\
        \leq&\E\Big[\big\langle\partial_x V(t_0,\xi^1,\mu_{t_0}^1)-\partial_x V(t_0,\xi^2,\mu_{t_0}^2),\xi^1-\xi^2\big\rangle+\lambda|\xi^1-\xi^2|^2\Big]\\
        &+\int_{t_0}^t\Big\{ C\E[|X_s^1-X_s^2|^2]+C\E[|X_s^1-X_s^2||\partial_x V_s^1-\partial_x V_s^2|]\\
        &\qquad-C_1\E[|\partial_x V_s^1-\partial_x V_s^2|^2]\Big\}\,ds,
    \end{aligned}
    \]
    where we use Assumption \ref{H regularity}(iii) (see also Remark \ref{remark H regularity}(ii)) in the last inequality. Since $V$ satisfies the displacement $\lambda$-monotonicity condition, applying Young's inequality, we obtain
    \[
    \begin{aligned}
        &\int_{t_0}^t \E[|\partial_x V(s,X_s^1,\mu_s^1)-\partial_x V(s,X_s^2,\mu_s^2)|^2]ds \\
        \leq& C\E\Big[\big\langle\partial_x V(\xi^1,\mu_{t_0}^1)-\partial_x V(\xi^2,\mu_{t_0}^2),\xi^1-\xi^2\big\rangle+\lambda|\xi^1-\xi^2|^2\Big]+C\int_{t_0}^t\E[|X_s^1-X_s^2|^2]\,ds\\
        \leq& C\E[|\xi^1-\xi^2|^2]+C\E[|\partial_x V(\xi^2,\mu_{t_0}^1)-\partial_x V(\xi^2,\mu_{t_0}^2)|^2]^{\frac12}\E[|\xi^1-\xi^2|^2]^{\frac12}\\
        &+C\int_{t_0}^t\E[|X_s^1-X_s^2|^2]\,ds.
    \end{aligned}
    \]
    Using the above estimate and combining with the SDE satisfied by $X^1,X^2$, we have
    \begin{equation}
    \label{disp difference mu}
    \begin{aligned}
        &\E[W_2^2(\mu_t^1,\mu_t^2)]\leq\E[|X_t^1-X_t^2|^2]\\
        \leq& C\E[|\xi^1-\xi^2|^2]+C\E[|\partial_x V(\xi^2,\mu_{t_0}^1)-\partial_x V(\xi^2,\mu_{t_0}^2)|^2]^{\frac12}\E[|\xi^1-\xi^2|^2]^{\frac12}.
    \end{aligned}
    \end{equation}
    Moreover, following the proof of \cite[Proposition 3.4]{bansil2025degenerate}, we can directly compute that
    \begin{equation}
    \label{disp difference Vx}
        \sup_{t\in[t_0,T]}\E\|\partial_x V(t,\cdot,\mu_t^1)-\partial_x V(t,\cdot,\mu_t^2)\|_{L^{\infty}}\leq C\sup_{t\in[t_0,T]}\E W_2(\mu_t^1,\mu_t^2).
    \end{equation}
    Plugging (\ref{disp difference Vx}) into (\ref{disp difference mu}), we obtain
    \begin{equation}
        \sup_{t\in[t_0,T]}\E W_2(\mu_t^1,\mu_t^2)\leq C\Big(\E[|\xi^1-\xi^2|^2]\Big)^{\frac12}.
    \end{equation}
    We can choose $\xi^1,\xi^2$ such that $W_2(\mu_{t_0}^1,\mu_{t_0}^2)=\E[|\xi^1-\xi^2|^2]^{\frac12}$. Then, from (\ref{disp difference Vx}) we conclude that
    \[
     \sup_{t\in[t_0,T]}\E\|\partial_x V(t,\cdot,\mu_t^1)-\partial_x V(t,\cdot,\mu_t^2)\|_{L^{\infty}}\leq C W_2(\mu_{t_0}^1,\mu_{t_0}^2),
    \]
    which implies the $W_2$-Lipschitz continuity of $\partial_x V$ at $t_0$.
\end{proof}

\subsection{The \texorpdfstring{$W_1$}{W1}-Lipschitz continuity under Lasry-Lions monotonicity condition}
\par The following theorem can be viewed as an analogue of Theorem \ref{W2 displacement} under the setting of Lasry-Lions monotonicity. However, in this case we can directly obtain the $W_1$-Lipschitz estimate.

\begin{theorem}
\label{W1 LL}
    Suppose $V$ is a classical solution to the master equation (\ref{master equation}) with bounded $\partial_x V,\partial_{xx}V$ and $\partial_{x\mu}V$. Let Assumptions \ref{fixed point}, \ref{G regularity} and \ref{H regularity}(i)(iii) hold. Assume further $V(t,\cdot,\cdot)$ satisfies the Lasry-Lions monotonicity condition (\ref{LL monotone}).
    Then $\partial_x V$ is uniformly Lipschitz continuous in $\mu$ under $W_1$, where the Lipschitz constant depends only on $d, T, \|\partial_x V\|_{L^{\infty}},\|\partial_{xx}V\|_{L^{\infty}}$, and $L_x^G,L_{\mu}^G$ in Assumption \ref{G regularity}, $L^H,C_1$ in Assumption \ref{H regularity}.
\end{theorem}
\begin{proof}
    Following the computation in Theorem \ref{propagation LL}, we have
    \begin{equation}
        \E[V(t,X_t^1,\mu_t^1)+V(t,X_t^2,\mu_t^2)-V(t,X_t^2,\mu_t^1)-V(t,X_t^1,\mu_t^2)]=I_1+I_2+I_3,
    \end{equation}
    where
    \[
    I_1:=\E[V(t_0,\xi^1,\mu_{t_0}^1)+V(t_0,\xi^2,\mu_{t_0}^2)-V(t_0,\xi^2,\mu_{t_0}^1)-V(t_0,\xi^1,\mu_{t_0}^2)],
    \]
    \[
    \begin{aligned}
        I_2:=&\int_{t_0}^t\E\Big[\big\langle\partial_x V(s,X_s^1,\mu_s^1)-\partial_x V(s,X_s^1,\mu_s^2),\partial_p\widehat{H}\big(X_s^1,\partial_x V(s,X_s^1,\mu_s^1),\rho_s^1\big)\big\rangle\\
        &\qquad-\widehat{H}\big(X_s^1,\partial_x V(s,X_s^1,\mu_s^1),\rho_s^1\big)+\widehat{H}\big(X_s^1,\partial_x V(s,X_s^1,\mu_s^2),\rho_s^1\big)\\
        &\qquad+\big\langle\partial_x V(s,X_s^2,\mu_s^2)-\partial_x V(s,X_s^2,\mu_s^1),\partial_p\widehat{H}\big(X_s^2,\partial_x V(s,X_s^2,\mu_s^2),\rho_s^2\big)\big\rangle\\
        &\qquad-\widehat{H}\big(X_s^2,\partial_x V(s,X_s^2,\mu_s^2),\rho_s^2\big)+\widehat{H}\big(X_s^2,\partial_x V(s,X_s^2,\mu_s^1),\rho_s^2\big)\Big]\,ds,
    \end{aligned}
    \]
    \[
    \begin{aligned}
        I_3:=&\int_{t_0}^t\E\Big[-\widehat{H}\big(X_s^1,\partial_x V(s,X_s^1,\mu_s^2),\rho_s^1\big)+\widehat{H}\big(X_s^1,\partial_x V(s,X_s^1,\mu_s^2),\rho_s^2\big)\\
        &\qquad-\widehat{H}\big(X_s^2,\partial_x V(s,X_s^2,\mu_s^1),\rho_s^2\big)+\widehat{H}\big(X_s^2,\partial_x V(s,X_s^2,\mu_s^1),\rho_s^1\big)\Big]\,ds.
    \end{aligned}
    \]
    We first have
    \begin{equation}
    \label{estimate I1}
    I_1=\E\Big[\int_0^1\big\langle\partial_x V(t_0,\xi_{\lambda},\mu_{t_0}^1)-\partial_x V(t_0,\xi_{\lambda},\mu_{t_0}^2),\xi^1-\xi^2\big\rangle\,d\lambda\Big],
    \end{equation}
    where $\xi_{\lambda}=\lambda\xi^1+(1-\lambda)\xi^2$. Using the concavity $\partial_{pp}\widehat{H}\leq-c_0 I_d$, we have
    \begin{equation}
    \label{estimate I2}
    \begin{aligned}
        I_2\leq&\int_{t_0}^t\Big( -\frac12 c_0\E[|\partial_x V(s,X_s^1,\mu_s^1)-\partial_x V(s,X_s^1,\mu_s^2)|^2\big]\\
        &-\frac12 c_0\E\big[|\partial_x V(s,X_s^2,\mu_s^1)-\partial_x V(s,X_s^2,\mu_s^2)|^2]\Big)\,ds.
    \end{aligned}
    \end{equation}
    Next, using the boundedness for $\partial_{xx}V$, second order derivatives of $\widehat{H}$ and $|\partial_{p\rho_2}\widehat{H}|\leq c_1$, we have the estimate
    \begin{equation}
    \label{estimate I3}
    \begin{aligned}
        I_3\leq&\int_{t_0}^t\Big\{C\big(\E[|X_s^1-X_s^2|]\big)^2+C\E[|X_s^1-X_s^2|]\E[|\partial_x V(s,X_s^1,\mu_s^1)-\partial_x V(s,X_s^2,\mu_s^2)|]\\
        &\qquad+C\E[|X_s^1-X_s^2|]\E[|\partial_x V(s,X_s^1,\mu_s^2)-\partial_x V(s,X_s^2,\mu_s^1)|]\\
        &\qquad+c_1\E[|\partial_x V(s,X_s^1,\mu_s^2)-\partial_x V(s,X_s^2,\mu_s^1)|]\\
        &\qquad\qquad\cdot\E[|\partial_x V(s,X_s^1,\mu_s^1)-\partial_x V(s,X_s^2,\mu_s^2)|]\Big\}\,ds\\
        \leq&\int_{t_0}^t\Big\{C\big(\E[|X_s^1-X_s^2|]\big)^2+(\frac{c_1}{2}+\epsilon)\Big(\E[|\partial_x V(s,X_s^1,\mu_s^1)-\partial_x V(s,X_s^2,\mu_s^1)|^2]\\
        &\qquad+\E[|\partial_x V(s,X_s^2,\mu_s^1)-\partial_x V(s,X_s^2,\mu_s^2)|^2]\Big)\Big\}\,ds
    \end{aligned}
    \end{equation}
    for some undetermined $\epsilon>0$, where we use Young's inequality in the last inequality.
    Since $V$ satisfies the Lasry-Lions monotonicity condition (\ref{LL monotone}), we know $I_1+I_2+I_3\geq 0$. Now combining with (\ref{estimate I1})-(\ref{estimate I3}), noting that $c_1<c_0$ and choosing $\epsilon$ small enough, we conclude that
    \begin{equation}
    \label{LL difference Vx2}
    \begin{aligned}
        &\int_{t_0}^t\E\Big[|\partial_x V(s,X_s^1,\mu_s^1)-\partial_x V(s,X_s^1,\mu_s^2)|^2+|\partial_x V(s,X_s^2,\mu_s^1)-\partial_x V(s,X_s^2,\mu_s^2)|^2\Big]\,ds\\
        \leq&C\int_{t_0}^t\big(\E[|X_s^1-X_s^2|]\big)^2\,ds\\
        &+C\E\Big[\int_0^1\big\langle\partial_x V(t_0,\xi_{\lambda},\mu_{t_0}^1)-\partial_x V(t_0,\xi_{\lambda},\mu_{t_0}^2),\xi^1-\xi^2\big\rangle\,d\lambda\Big].
    \end{aligned}
    \end{equation} 
    Using the SDE satisfied by $X^1,X^2$, we have the estimate
    \begin{equation}
    \begin{aligned}
        &\E[|X_t^1-X_t^2|]\\
        \leq&\E[|\xi^1-\xi^2|]+C\int_{t_0}^t\E[|\partial_x V(s,X_s^1,\mu_s^1)-\partial_x V(s,X_s^2,\mu_s^2)|]\,ds\\
        \leq&\E[|\xi^1-\xi^2|]+C\int_{t_0}^t\big(\E[|\partial_x V(s,X_s^1,\mu_s^1)-\partial_x V(s,X_s^1,\mu_s^2)|]+\E[|X_s^1-X_s^2|]\big)\,ds.\\      
    \end{aligned}
    \end{equation}
    Using (\ref{LL difference Vx2}) and Gronwall's inequality, we obtain
    \begin{equation}\label{LL difference mu}
    \begin{aligned}
        &\big(\E[W_1(\mu_t^1,\mu_t^2)]\big)^2\leq\big(\E[|X_t^1-X_t^2|]\big)^2\\
        \leq&\big(\E[|\xi^1-\xi^2|]\big)^2+C\E\Big[\int_0^1\big\langle\partial_x V(t_0,\xi_{\lambda},\mu_{t_0}^1)-\partial_x V(t_0,\xi_{\lambda},\mu_{t_0}^2),\xi^1-\xi^2\big\rangle\,d\lambda\Big].
    \end{aligned}
    \end{equation}
    Similar to \cite[Proposition 3.4]{bansil2025degenerate}, we can derive the estimate
    \begin{equation}
    \label{LL difference Vx}
        \sup_{t\in[t_0,T]}\E\|\partial_x V(t,\cdot,\mu_t^1)-\partial_x V(t,\cdot,\mu_t^2)\|_{L^{\infty}}\leq C\sup_{t\in[t_0,T]}\E W_1(\mu_t^1,\mu_t^2).  
    \end{equation}
    Plugging (\ref{LL difference Vx}) into (\ref{LL difference mu}), then applying Young's inequality, we obtain
    \[
    \sup_{t\in[t_0,T]}\E[W_1(\mu_t^1,\mu_t^2)]\leq C\E[|\xi^1-\xi^2|].
    \]
    We can choose $\xi^1,\xi^2$ such that $W_1(\mu_{t_0}^1,\mu_{t_0}^2)=\E[|\xi^1-\xi^2|]$. Then, from (\ref{LL difference Vx}) we conclude that
    \[
    \sup_{t\in[t_0,T]}\E\|\partial_x V(t,\cdot,\mu_t^1)-\partial_x V(t,\cdot,\mu_t^2)\|_{L^{\infty}}\leq CW_1(\mu_{t_0}^1,\mu_{t_0}^2),
    \]
    which implies the $W_1$-Lipschitz continuity of $\partial_x V$ at $t_0$.    
\end{proof}

\section{Global well-posedness of the master equation}
In this section we shall establish the global well-posedness of the master equation (\ref{master equation}). Note that the global well-posedness requires $W_1$-Lipschitz continuity. Hence we shall derive the $W_1$-Lipschitz continuity of $\partial_x V$ in $\mu$ from the $W_2$-Lipschitz continuity established above under the displacement $\lambda$-monotonicity condition. This can be achieved by a pointwise representation for Wasserstein derivatives developed in \cite{mou2020wellposedness}. To derive a representation formula for $\partial_{x\mu}V$, we first consider the following McKean-Vlasov FBSDE: given $t_0\in[0,T]$ and $\xi\in\mathbb{L}^2(\mathcal{F}_{t_0})$,
\begin{equation}
\label{xi}
\left\{
\begin{aligned}
    &X_t^{\xi}=\xi+\int_{t_0}^t\partial_p\widehat{H}(X_s^{\xi},\nabla Y_s^{\xi},\rho_s)\,ds+B_t^{t_0}+\beta B_t^{0,t_0},\\
    &\nabla Y_t^{\xi}=\partial_x G(X_T^{\xi},\mu_T)+\int_t^T\partial_x\widehat{H}(X_s^{\xi},\nabla Y_s^{\xi},\rho_s)\,ds\\
    &\qquad\qquad-\int_t^T \nabla Z_s^{\xi}\cdot dB_s-\int_t^T \nabla Z_s^{0,\xi}\cdot dB_s^0,\\
    &\mu_t:=\mathcal{L}_{X_t^{\xi}|\mathcal{F}_t^0},\quad\rho_t:=\rho_t^{\xi}:=\mathcal{L}_{(X_t^{\xi},\nabla Y_t^{\xi})|\mathcal{F}_t^0}.
\end{aligned}
\right.
\end{equation}
$\nabla Y^{\xi}$ is related to the master equation (\ref{master equation}) as $\nabla Y_t^{\xi}=\partial_x V(t,X_t^{\xi},\mathcal{L}_{X_t^{\xi}|\mathcal{F}_t^0})$. Given $\rho$ as above and $x\in\mathbb{R}^d$, we consider the following FBSDE:
\begin{equation}
\label{xi,x}
\left\{
\begin{aligned}
    &X_t^{\xi,x}=x+\int_{t_0}^t\partial_p\widehat{H}(X_s^{\xi,x},\nabla Y_s^{\xi,x},\rho_s)\,ds+B_t^{t_0}+\beta B_t^{0,t_0},\\
    &\nabla Y_t^{\xi,x}=\partial_x G(X_T^{\xi,x},\mu_T)+\int_t^T\partial_x\widehat{H}(X_s^{\xi,x},\nabla Y_s^{\xi,x},\rho_s)\,ds\\
    &\qquad\qquad-\int_t^T \nabla Z_s^{\xi,x}\cdot dB_s-\int_t^T \nabla Z_s^{0,\xi,x}\cdot dB_s^0.\\
\end{aligned}
\right.
\end{equation}
\par Next we consider the following FBSDEs on $[t_0,T]$:
\begin{equation}
\label{nablak xi,x}
\left\{
\begin{aligned}
    \nabla_k X_t^{\xi,x}=&e_k+\int_{t_0}^t\Big[\partial_{xp}\widehat{H}(X_s^{\xi,x},\nabla Y_s^{\xi,x},\rho_s)\nabla_k X_s^{\xi,x}\\
    &\qquad\qquad+\partial_{pp}\widehat{H}(X_s^{\xi,x},\nabla Y_s^{\xi,x},\rho_s)\nabla_k^2 Y_s^{\xi,x}\Big]\,ds,\\
    \nabla_k^2 Y_t^{\xi,x}=&\partial_{xx}G(X_T^{\xi,x},\mu_T)\nabla_k X_T^{\xi,x}-\int_t^T \nabla_k^2 Z_s^{\xi,x}\cdot dB_s-\int_t^T \nabla_k^2 Z_s^{0,\xi,x}\cdot dB_s^0\\
    &+\int_t^T\Big[\partial_{xx}\widehat{H}(X_s^{\xi,x},\nabla Y_s^{\xi,x},\rho_s)\nabla_k X_s^{\xi,x}\\
    &\qquad\qquad+\partial_{px}\widehat{H}(X_s^{\xi,x},\nabla Y_s^{\xi,x},\rho_s)\nabla_k^2 Y_s^{\xi,x}\Big]\,ds,
\end{aligned}
\right.
\end{equation}
which is a formal differentiation of (\ref{xi,x}) with respect to $x_k$, and
\small\begin{equation}
\label{nablak mathcal xi,x}
\left\{
\begin{aligned}
    \nabla_k\mathcal{X}_t^{\xi,x}=&\int_{t_0}^t\Big\{\partial_{xp}\widehat{H}(X_s^{\xi},\nabla Y_s^{\xi},\rho_s)\nabla_k\mathcal{X}_s^{\xi,x}+\partial_{pp}\widehat{H}(X_s^{\xi},\nabla Y_s^{\xi},\rho_s)\nabla_k^2 \mathcal{Y}_s^{\xi,x}\\
    &\qquad+\tilde{\E}_{\mathcal{F}_s}\big[\partial_{\rho_1 p}\widehat{H}(X_s^{\xi},\nabla Y_s^{\xi},\rho_s,\tilde{X}_s^{\xi,x},\nabla\tilde{Y}_s^{\xi,x})\nabla_k\tilde{X}_s^{\xi,x}\\
    &\qquad+\partial_{\rho_2 p}\widehat{H}(X_s^{\xi},\nabla Y_s^{\xi},\rho_s,\tilde{X}_s^{\xi,x},\nabla\tilde{Y}_s^{\xi,x})\nabla_k^2\tilde{Y}_s^{\xi,x}\\
    &\qquad+\partial_{\rho_1 p}\widehat{H}(X_s^{\xi},\nabla Y_s^{\xi},\rho_s,\tilde{X}_s^{\xi,x},\nabla\tilde{Y}_s^{\xi,x})\nabla_k\tilde{\mathcal{X}}_s^{\xi,x}\\
    &\qquad+\partial_{\rho_2 p}\widehat{H}(X_s^{\xi},\nabla Y_s^{\xi},\rho_s,\tilde{X}_s^{\xi,x},\nabla\tilde{Y}_s^{\xi,x})\nabla_k^2\tilde{\mathcal{Y}}_s^{\xi,x}\big]\Big\}\,ds,\\
    \nabla_k^2\mathcal{Y}_t^{\xi,x}=&\partial_{xx}G(X_T^{\xi},\mu_T)\nabla_k\mathcal{X}_T^{\xi,x}-\int_t^T\nabla_k^2\mathcal{Z}_s^{\xi,x}\cdot dB_s-\int_t^T\nabla_1^2\mathcal{Z}_s^{0,\xi,x}\cdot dB_s^0\\
    &+\tilde{\E}_{\mathcal{F}_T}\big[\partial_{\mu x}G(X_T^{\xi},\mu_T,\tilde{X}_T^{\xi,x})\nabla_k\tilde{X}_T^{\xi,x}+\partial_{\mu x}G(X_T^{\xi},\mu_T,\tilde{X}_T^{\xi})\nabla_k\tilde{\mathcal{X}}_T^{\xi,x}\big]\\
    &+\int_t^T\Big\{\partial_{xx}\widehat{H}(X_s^{\xi},\nabla Y_s^{\xi},\rho_s)\nabla_k\mathcal{X}_s^{\xi,x}+\partial_{px}\widehat{H}(X_s^{\xi},\nabla Y_s^{\xi},\rho_s)\nabla_k^2\mathcal{Y}_s^{\xi,x}\\
    &\qquad+\tilde{\E}_{\mathcal{F}_s}\big[\partial_{\rho_1 x}\widehat{H}(X_s^{\xi},\nabla Y_s^{\xi},\rho_s,\tilde{X}_s^{\xi,x},\nabla\tilde{Y}_s^{\xi,x})\nabla_k\tilde{X}_s^{\xi,x}\\
    &\qquad+\partial_{\rho_2 x}\widehat{H}(X_s^{\xi},\nabla Y_s^{\xi},\rho_s,\tilde{X}_s^{\xi,x},\nabla\tilde{Y}_s^{\xi,x})\nabla_k^2\tilde{Y}_s^{\xi,x}\\
    &\qquad+\partial_{\rho_1 x}\widehat{H}(X_s^{\xi},\nabla Y_s^{\xi},\rho_s,\tilde{X}_s^{\xi,x},\nabla\tilde{Y}_s^{\xi,x})\nabla_k\tilde{\mathcal{X}}_s^{\xi,x}\\
    &\qquad+\partial_{\rho_2 x}\widehat{H}(X_s^{\xi},\nabla Y_s^{\xi},\rho_s,\tilde{X}_s^{\xi,x},\nabla\tilde{Y}_s^{\xi,x})\nabla_k^2\tilde{\mathcal{Y}}_s^{\xi,x}\big]\Big\}\,ds.
\end{aligned}
\right.
\end{equation}\normalsize
Then, the representation formula for $\partial_{x\mu}V$ is given by the following FBSDE:
\small\begin{equation}
\label{representation FBSDE}
\left\{
\begin{aligned}
    \nabla_{\mu_k}X_t^{\xi,x,\tilde{x}}=&\int_{t_0}^t\Big\{\partial_{xp}\widehat{H}(X_s^{\xi,x},\nabla Y_s^{\xi,x},\rho_s)\nabla_{\mu_k}X_s^{\xi,x,\tilde{x}}\\
    &\qquad+\partial_{pp}\widehat{H}(X_s^{\xi,x},\nabla Y_s^{\xi,x},\rho_s)\nabla_{\mu_k}^2 Y_s^{\xi,x,\tilde{x}}\\
    &\qquad+\tilde{\E}_{\mathcal{F}_s}\big[\partial_{\rho_1 p}\widehat{H}(X_s^{\xi,x},\nabla Y_s^{\xi,x},\rho_s,\tilde{X}_s^{\xi,\tilde{x}},\nabla\tilde{Y}_s^{\xi,\tilde{x}})\nabla_k\tilde{X}_s^{\xi,\tilde{x}}\\
    &\qquad+\partial_{\rho_2 p}\widehat{H}(X_s^{\xi,x},\nabla Y_s^{\xi,x},\rho_s,\tilde{X}_s^{\xi,\tilde{x}},\nabla\tilde{Y}_s^{\xi,\tilde{x}})\nabla_k^2\tilde{Y}_s^{\xi,\tilde{x}}\\
    &\qquad+\partial_{\rho_1 p}\widehat{H}(X_s^{\xi,x},\nabla Y_s^{\xi,x},\rho_s,\tilde{X}_s^{\xi,\tilde{x}},\nabla\tilde{Y}_s^{\xi,\tilde{x}})\nabla_k\tilde{\mathcal{X}}_s^{\xi,\tilde{x}}\\
    &\qquad+\partial_{\rho_2 p}\widehat{H}(X_s^{\xi,x},\nabla Y_s^{\xi,x},\rho_s,\tilde{X}_s^{\xi,\tilde{x}},\nabla\tilde{Y}_s^{\xi,\tilde{x}})\nabla_k^2\tilde{\mathcal{Y}}_s^{\xi,\tilde{x}}\big]\Big\}\,ds\\
    \nabla_{\mu_k}^2 Y_t^{\xi,x,\tilde{x}}=&\partial_{xx}G(X_T^{\xi,x},\mu_T)\nabla_{\mu_k}X_T^{\xi,x,\tilde{x}}+\tilde{\E}_{\mathcal{F}_T}\big[\partial_{\mu x}G(X_T^{\xi,x},\mu_T,\tilde{X}_T^{\xi,\tilde{x}})\nabla_k\tilde{X}_T^{\xi,\tilde{x}}\\
    &+\partial_{\mu x}G(X_T^{\xi,x},\mu_T,\tilde{X}_T^{\xi})\nabla_k\tilde{\mathcal{X}}_T^{\xi,\tilde{x}}\big]\\
    &-\int_t^T\nabla_{\mu_k}^2 Z_s^{\xi,x,\tilde{x}}\cdot dB_s-\int_t^T\nabla_{\mu_k}^2 Z_s^{0,\xi,x,\tilde{x}}\cdot dB_s^0\\
    &+\int_t^T\Big\{\partial_{xx}\widehat{H}(X_s^{\xi,x},\nabla Y_s^{\xi,x},\rho_s)\nabla_{\mu_k}X_s^{\xi,x,\tilde{x}}\\
    &\qquad+\partial_{px}\widehat{H}(X_s^{\xi,x},\nabla Y_s^{\xi,x},\rho_s)\nabla_{\mu_k}^2 Y_s^{\xi,x,\tilde{x}}\\
    &\qquad+\tilde{\E}_{\mathcal{F}_s}\big[\partial_{\rho_1 x}\widehat{H}(X_s^{\xi,x},\nabla Y_s^{\xi,x},\rho_s,\tilde{X}_s^{\xi,\tilde{x}},\nabla\tilde{Y}_s^{\xi,\tilde{x}})\nabla_k\tilde{X}_s^{\xi,\tilde{x}}\\
    &\qquad+\partial_{\rho_2 x}\widehat{H}(X_s^{\xi,x},\nabla Y_s^{\xi,x},\rho_s,\tilde{X}_s^{\xi,\tilde{x}},\nabla\tilde{Y}_s^{\xi,\tilde{x}})\nabla_k^2\tilde{Y}_s^{\xi,\tilde{x}}\\
    &\qquad+\partial_{\rho_1 x}\widehat{H}(X_s^{\xi,x},\nabla Y_s^{\xi,x},\rho_s,\tilde{X}_s^{\xi,\tilde{x}},\nabla\tilde{Y}_s^{\xi,\tilde{x}})\nabla_k\tilde{\mathcal{X}}_s^{\xi,\tilde{x}}\\
    &\qquad+\partial_{\rho_2 x}\widehat{H}(X_s^{\xi,x},\nabla Y_s^{\xi,x},\rho_s,\tilde{X}_s^{\xi,\tilde{x}},\nabla\tilde{Y}_s^{\xi,\tilde{x}})\nabla_k^2\tilde{\mathcal{Y}}_s^{\xi,\tilde{x}}\big]\Big\}\,ds.
\end{aligned}\right.
\end{equation}\normalsize
\begin{theorem}
\label{W1 Lipschitz}
    Let Assumptions \ref{fixed point}, \ref{G regularity} and \ref{H regularity}(i)(ii) hold. Then there exists $\delta>0$, depending only on $d$, $L_x^G$ in Assumption \ref{G regularity}, $\tilde{L}_{\mu}^G$ in Remark \ref{G regularity remark}, $L^H(L_x^V)$ in Assumption \ref{H regularity}(i) and Proposition \ref{Vxx bounded}, such that if $T-t_0\leq\delta$, the following statements hold.\\
    (i) All the FBSDEs (\ref{xi})-(\ref{representation FBSDE}) are well-posed on $[t_0,T]$, for any $x\in\R^d,\,\mu\in\mathcal{P}_2(\R^d)$ and $\xi\in\mathbb{L}^2(\mathcal{F}_{t_0};\mu)$.\\
    (ii) Define $\vec{U}(t_0,x,\mu):=\nabla Y_{t_0}^{\xi,x}$. Then we have the pointwise representation formula:
    \begin{equation}
    \label{representation formula}
        \partial_{\mu_k}\vec{U}(t_0,x,\mu,\tilde{x})=\nabla_{\mu_k}^2 Y_{t_0}^{\xi,x,\tilde{x}}.
    \end{equation}
    Moreover, $\partial_{\mu_k}\vec{U}(t_0,x,\mu,\tilde{x})$ is uniformly bounded, where the bound depends only on $d,L^H,L_{x}^G$ and $L_{\mu}^G$.\\
    (iii) The following decoupled FBSDE
    \begin{equation}
    \label{fbsde define V}
    \left\{
    \begin{aligned}
        &X_t^x=x+B_t^{t_0}+\beta B_t^{0,t_0},\\
        &Y_t^{x,\xi}=G(X_T^x,\mu_T)+\int_t^T\widehat{H}(X_s^x,\vec{U}(s,X_s^x,\mu_s),\rho_s)\,ds\\
        &\qquad\qquad-\int_t^T Z_s^{x,\xi}\cdot dB_s-\int_t^T Z_s^{0,x,\xi}\cdot dB_s^0
    \end{aligned}
    \right.
    \end{equation}
    is well-posed on $[t_0,T]$ for any $x\in\mathbb{R}^d$. Define $V(t_0,x,\mu):=Y_{t_0}^{x,\xi}$. Then $V$ is the unique classical solution to the master equation (\ref{master equation}) on $[t_0,T]$ with bounded $\partial_x V,\partial_{xx}V,\partial_{x\mu}V$. Moreover, we have $\partial_x V=\vec{U}$.
\end{theorem}
\begin{proof}
    The proof of (i)-(ii) is lengthy but similar to that of \cite[Proposition 6.2]{mou2022displacement}, so we omit here.\\
    (iii) From the FBSDE (\ref{xi,x}) for $\nabla Y^{\xi,x}$, we can show that $\vec{U},\partial_x\vec{U}$ defined in (ii) are uniformly bounded. Then the FBSDE (\ref{fbsde define V}) is well-posed. Next we need to verify that $V$ is the unique solution to the master equation (\ref{master equation}) and $\partial_x V=\vec{U}$. The proof is similar to that of \cite[Proposition 5.2]{mou2022anti}. {However, since we work with less regular data $\widehat{H}$ and $G$, we cannot directly differentiate the master equation, so a mollification argument is necessary.}
    \par To this end, we employ the smooth approximation constructed in \cite{cosso2023smooth}. Since $G\in\mathcal{C}^2(\R^d\times\mathcal{P}_2(\R^d))$ and $\widehat{H}\in\mathcal{C}^2(\R^{2d}\times\mathcal{P}_2(\R^{2d}))$, we can construct smooth sequences $\{G_n\}$ and $\{\widehat{H}_n\}$ such that $G_n,\widehat{H}_n$, as well as their first and second order derivatives, converge uniformly those of $\widehat{H},G$. By replacing the data $(G,\widehat{H})$ in FBSDE (\ref{xi,x}) with $(G_n,\widehat{H}_n)$, we can define $\vec{U}_n$ as in (ii). Similarly, by replacing the data $(G,\widehat{H},\vec{U})$ in FBSDE (\ref{fbsde define V}) with $(G_n,\widehat{H}_n,\vec{U}_n)$, we can define $V_n$ as in (iii). Since $U_n,V_n$ are smooth, we can apply It\^{o}'s formula to $V_n$ to deduce that $V_n$ satisfies the following equation:
    \begin{equation}\label{pde V_n}
    \begin{gathered}
        \partial_t V_n+\frac{\widehat{\beta}^2}{2}\tr(\partial_{xx}V_n)+\widehat{H}_n\big(x,\vec{U}_n(t,x,\mu),\mathcal{L}_{(\xi,\vec{U}_n(t,\xi,\mu))}\big)\\
        +\tr\Big(\bar{\tilde{\E}}\Big[\frac{\widehat{\beta}^2}{2}\partial_{\tilde{x}\mu}V_n(t,x,\mu,\tilde{\xi})+\beta^2\partial_{x\mu}V_n(t,x,\mu,\tilde{\xi})+\frac{\beta^2}{2}\partial_{\mu\mu}V_n(t,x,\mu,\bar{\xi},\tilde{\xi})\\+\partial_{\mu}V_n(t,x,\mu,\tilde{\xi})\cdot\partial_p\widehat{H}_n\big(\tilde{\xi},\vec{U}_n(t,\tilde{\xi},\mu),\mathcal{L}_{(\xi,\vec{U}_n(t,\xi,\mu))}\big)\Big]\Big)=0,\\
        V_n(T,x,\mu)=G_n(x,\mu).
    \end{gathered}
    \end{equation}
    Differentiating the above equation with respect to $x$, we obtain the equation satisfied by $\vec{U}_n':=\partial_x V_n$:
    \begin{equation}
    \label{pde U_n'}
    \begin{gathered}
        \partial_t \vec{U}_n'+\frac{\widehat{\beta}^2}{2}\tr(\partial_{xx}\vec{U}_n')+\partial_x\widehat{H}_n\big(x,\vec{U}_n(t,x,\mu),\mathcal{L}_{(\xi,\vec{U}_n(t,\xi,\mu))}\big)+\partial_p\widehat{H}_n\cdot\partial_x\vec{U}_n(t,x,\mu)\\
        +\tr\Big(\bar{\tilde{\E}}\Big[\frac{\widehat{\beta}^2}{2}\partial_{\tilde{x}\mu}\vec{U}_n'(t,x,\mu,\tilde{\xi})+\beta^2\partial_{x\mu}\vec{U}_n'(t,x,\mu,\tilde{\xi})+\frac{\beta^2}{2}\partial_{\mu\mu}\vec{U}_n'(t,x,\mu,\bar{\xi},\tilde{\xi})\\
        +\partial_{\mu}\vec{U}_n'(t,x,\mu,\tilde{\xi})\cdot\partial_p\widehat{H}_n\big(\tilde{\xi},\vec{U}_n(t,\tilde{\xi},\mu),\mathcal{L}_{(\xi,\vec{U}_n(t,\xi,\mu))}\big)\Big]\Big)=0,\\
        \vec{U}_n'(T,x,\mu)=\partial_x G_n(x,\mu).
    \end{gathered}
    \end{equation}
    On the other hand, applying It\^{o}'s formula to $\vec{U}_n$, we can verify that $\vec{U}_n$ also satisfies the equation (\ref{pde U_n'}) above. Denote $\Delta \vec{U}_n:=\vec{U}_n-\vec{U}_n'$. Then $\Delta \vec{U}_n$ satisfies the equation
    \begin{equation}
    \begin{gathered}
        \partial_t \Delta \vec{U}_n+\frac{\widehat{\beta}^2}{2}\tr(\partial_{xx}\Delta \vec{U}_n)+\\
        \tr\Big(\bar{\tilde{\E}}\Big[\frac{\widehat{\beta}^2}{2}\partial_{\tilde{x}\mu}\Delta \vec{U}_n(t,x,\mu,\tilde{\xi})+\beta^2\partial_{x\mu}\Delta \vec{U}_n(t,x,\mu,\tilde{\xi})        +\frac{\beta^2}{2}\partial_{\mu\mu}\Delta \vec{U}_n(t,x,\mu,\bar{\xi},\tilde{\xi})\\
        +\partial_{\mu}\Delta \vec{U}_n(t,x,\mu,\tilde{\xi})\cdot\partial_p\widehat{H}_n\big(\tilde{\xi},\vec{U}_n(t,\tilde{\xi},\mu),\mathcal{L}_{(\xi,\vec{U}_n(t,\xi,\mu))}\big)\Big]\Big)=0
    \end{gathered}
    \end{equation}
    with terminal condition $\Delta\vec{U}_n(T,x,\mu)=0$. By It\^{o}'s formula, the above equation is connected to the following system for $t\in[t_0,T]$,
    \begin{equation}
    \left\{\begin{aligned}
        &X_t^{\xi}=\xi+\int_{t_0}^t\partial_p\widehat{H}_n\big(X_s^{\xi},\vec{U}_n(s,X_s^{\xi},\mathcal{L}_{X_s^{\xi}|\mathcal{F}_s^0}),\mathcal{L}_{\big(X_s^{\xi},\vec{U}_n(s,X_s^{\xi},\mathcal{L}_{X_s^{\xi}|\mathcal{F}_s^0})\big)|\mathcal{F}_s^0}\big)\,ds\\
        &\qquad\quad+B_t^{t_0}+\beta\,B_t^{0,t_0},\\
        &X_t^x=x+B_t^{t_0}+\beta\,B_t^{0,t_0},\\
        &\Delta Y_t^{x,\xi}=0
    \end{aligned}\right. 
    \end{equation}
    with the relation $\Delta Y_t^{x,\xi}=\Delta\vec{U}_n(t,X_t^x,\mathcal{L}_{X_t^{\xi}|\mathcal{F}_t^0})$. Therefore, we can conclude that $\vec{U}_n=\vec{U}_n'=\partial_x V_n$. Plugging this into (\ref{pde V_n}), we see that $V_n$ satisfies the master equation with mollified data $G_n,\widehat{H}_n$. Finally, using the convergence properties of the mollifiers $G_n,\widehat{H}_n$ and the stability of the FBSDE (\ref{fbsde define V}), we conclude that $V$ is a classical solution to the master equation (\ref{master equation}). The uniqueness follows from the well-posedness of FBSDEs (\ref{xi,x})(\ref{fbsde define V}).
\end{proof}
\par With the previous results, we can follow the approach of \cite{chassagneux2014probabilistic,mou2020wellposedness,mou2022anti} to establish the global well-posedness of the master equation (\ref{master equation}). Once the local well-posedness is obtained, with an a priori uniform estimate for $\partial_{x\mu}V$, the existence and uniqueness of the classical solution follows from an induction argument.
\begin{theorem}
\label{global well-posedness}
     Suppose Assumptions \ref{fixed point}, \ref{G regularity}, \ref{H regularity}, and either of the following monotonicity conditions hold:\\
     (i) $\widehat{H}$ satisfies Assumption \ref{assumption H LL} and $G$ satisfies Lasry-Lions monotonicity condition (\ref{LL monotone});\\
     (ii) $\widehat{H}$ satisfies Assumption \ref{assumption H disp} and $G$ satisfies displacement $\lambda$-monotonicity condition (\ref{disp monotone}).\\
     Then the master equation (\ref{master equation}) admits a unique classical solution $V$ with bounded $\partial_x V,\partial_{xx}V$ and $\partial_{x\mu}V$.
\end{theorem}
\begin{proof}
    \par We only prove the case (ii): the pair of displacement $\lambda$-monotonicity conditions holds. The other case follows from the same argument with parallel theorems.
    \par We first show the existence. By Theorems \ref{propagation disp} and \ref{W2 displacement}, let $\tilde{L}_{\mu}^V$ be the a priori uniform $W_2$-Lipschitz constant for $\partial_x V$ with respect to $\mu$. Let $\delta>0$ be the constant established in Theorem \ref{W1 Lipschitz}, with $L_x^G$ replaced by $L_x^V$ and $\tilde{L}_{\mu}^G$ replaced by $\tilde{L}_{\mu}^V$. Let $0=T_0<T_1<\ldots<T_n=T$ be a partition of $[0,T]$ such that for $i=1,\cdots,n$, $T_i-T_{i-1}\leq\frac{\delta}{2}$.
    \par Since $T_n-T_{n-2}\leq\delta$, by Theorem \ref{W1 Lipschitz}, the master equation (\ref{master equation}) with terminal condition $G$ admits a unique classical solution on $[T_{n-2},T_n]$. By Theorems \ref{propagation disp} and \ref{W2 displacement}, $\partial_x V(t,\cdot,\cdot)$ is uniformly Lipschitz continuous in $\mu$ under $W_2$ for $t\in[T_{n-2},T_n]$, with the Lipschitz constant $\tilde{L}_{\mu}^V$. Moreover, by Theorem \ref{W1 Lipschitz}(ii), $\partial_{x}V(T_{n-1},\cdot,\cdot)$ is also uniformly Lipschitz continuous in $\mu$ under $W_1$. We next consider the master equation on $[T_{n-3},T_{n-1}]$ with terminal condition $V(T_{n-1},\cdot,\cdot)$. Note that $\partial_{x}V(T_{n-1},\cdot,\cdot)$ has the same uniform Lipschitz continuity with constants $L_{x}^V$ and $\tilde{L}_{\mu}^V$ again. Hence we can apply Theorem \ref{W1 Lipschitz} to obtain a classical solution $V$ on $[T_{n-3},T_{n-1}]$. This extends the classical solution of the master equation to $[T_{n-3},T_n]$. By repeating the above steps finite times, we obtain a classical solution $V$ to the master equation (\ref{master equation}) on $[0,T]$.
    \par The global uniqueness follows from the local uniqueness result. Specifically, suppose $V$ and $V'$ are two classical solutions to the master equation. By comparing them backward in time, starting from $[T_{n-1},T_n]$, the local uniqueness ensures that $V$ and $V'$ coincide on each small time interval. This implies the uniqueness on the entire interval $[0,T]$.
\end{proof}

\section*{Acknowledgments}Shuhui Liu was supported in part by Hong Kong RGC Grant P0031382/S-ZG9U and the Postdoc Matching Fund of PolyU Grant 1-W32B; Xintian Liu was supported in part by Hong Kong RGC Grant GRF 11311422 and Hong Kong RGC Grant GRF 11303223; Chenchen Mou was partially supported by NSFC grant 12522122, Hong Kong RGC Grant GRF 11311422 and Hong Kong RGC Grant GRF 11303223; and Defeng Sun was supported by the Research Center for Intelligent Operations Research and Hong Kong RGC Senior Research Fellow Scheme No. SRFS2223-5S02. 

\bibliographystyle{plain}
\bibliography{references}
\end{document}